\documentclass{amsart}

\usepackage{latexsym, amssymb, amsmath,  amssymb, float}
\usepackage{psfrag}
\usepackage[latin1]{inputenc}
\usepackage{epsf,graphicx}
\usepackage{amssymb,latexsym}
\usepackage[T1]{fontenc}
\usepackage{psfrag}
\usepackage{subfigure}

\numberwithin{equation}{section}

\theoremstyle{plain} 
\newtheorem{theorem}{Theorem}[section]
\newtheorem{proposition}[theorem]{Proposition}

\newtheorem{corollary}[theorem]{Corollary}

\newtheorem{notation}[theorem]{Notation}

\usepackage{amscd,amsfonts}
\usepackage{amsfonts}
\usepackage{ upref}
\usepackage{subfigure,  verbatim}

 \def\eps{{\epsilon}}
\def\C{{\mathbb C}}

\def\R{{\mathbb R}}
\def\S{{\mathbb S}}
\def\H{{\mathbb H}}
\def\CP{{\mathbb C\mathbb P}}

\def\ov{\overline}

\newcommand{\pl}{\parallel}

\title[The bifurcation diagram of cubic polynomial vector fields]{The bifurcation diagram of cubic polynomial vector fields on $\CP^1$}
\author{
C.~Rousseau}\address{Department of mathematics and statistics\\
University of Montreal\\
C.P. 6128, succ. centre-ville\\
Montreal, Quebec, H3C 3J7\\
Canada}
\email{rousseac@dms.umontreal.ca}
\thanks{This
work is supported by NSERC in Canada.}

\begin{document}
\maketitle

\begin{abstract} In this paper we give the bifurcation diagram of the family of cubic vector fields $\dot z=z^3+ \eps_1z+\eps_0$ for $z\in \CP^1$, depending on the values of $\eps_1,\eps_0\in\C$. The bifurcation diagram is in $\R^4$, but its conic structure allows describing it for parameter values in $\S^3$. There are two open simply connected regions of structurally stable vector fields separated by surfaces corresponding to bifurcations of homoclinic connections between two separatrices of the pole at infinity. These branch from the codimension 2 curve of double singular points.  We also explain the bifurcation of homoclinic connection in terms of the description of Douady and Sentenac of polynomial vector fields.  \end{abstract}

\section{Introduction} 

This study of polynomial vector fields  
\begin{equation} 
\dot z = z^{k+1} + \sum_{j=0}^{k-1} \eps_jz^j, \qquad z\in \CP^1,\label{pol_vf}\end{equation} 
on $\CP^1$ was initiated by Douady and Sentenac in \cite{DS}. The point at infinity is a pole of order $k-1$ with $2k$ separatrices. In their long paper, Douady and Sentenac describe how the phase portrait can be obtained from the position of the separatrices of $\infty$. This allows them describing at length the geometry of the generic vector fields in this family, which are those for which there are no homoclinic connections between any two separatrices of $\infty$. They also show that there are exactly $C(k)= \frac{\binom{2k}{k}}{k+1}$  connected components of generic vector fields in the parameter space 
$\eps= (\eps_{k-1}, \dots, \eps_0),$ and that these are simply connected. ($C(k)$ is the $k$-th Catalan number.)
Their proof is extremely ingenious and uses that any generic polynomial vector field \eqref{pol_vf} is completely determined (possibly up to some symmetry by a rotation of order $k$) by:
\begin{itemize}
\item A topological invariant describing how the separatrices of infinity end generically at finite singular points, thus dividing $\CP^1$ in $k$ connected components. The union of the separatrices is called the \emph{connecting graph}.
\item An analytic invariant given by $k$ numbers $\tau_1, \dots, \tau_k\in \H$, where $\H\subset \C$ is the upper half-plane, and essentially equivalent to the collection of eigenvalues at each singular point. These numbers correspond to \lq\lq complex travel times\rq\rq\ along curves joining sectors to sectors at $\infty$, without cutting the connecting graph (see details below).  \end{itemize}

The results of Douady and Sentenac were further generalized to non generic vector fields by Branner and Dias \cite{BD}, who introduced invariants for non generic vector fields, but the paper \cite{BD} does not describe the bifurcation diagram. 

One of the goals of Douady and Sentenac was to develop tools for analyzing the unfoldings of parabolic points of germs of diffeomorphims on $(\C,0)$ (i.e. fixed points with multiplier equal to $1$). One of the first attempts to understand the unfolding of a codimension $1$ parabolic point goes back to Martinet \cite{M}, where he considered unfoldings for parameter values for which the unfolded fixed points are hyperbolic. He could show that the comparison of the two linearizing charts at each fixed point converges to the Ecalle-Voronin modulus of the parabolic point: indeed the domains of linearization tend to the sectors used in the definition of the Ecalle-Voronin modulus. But, until the thesis of Lavaurs \cite{L}, no method would allow studying the sector of parameter values containing the values for which the unfolded fixed points have multipliers on the unit circle. There, a new idea was introduced, namely to unfold the sectors allowing to define the Ecalle-Voronin modulus into  sectors with vertices at the two singular points on which almost unique changes of coordinates to the normal form exist. Lavaurs' thesis covered a sector of parameter values complementary to the one studied by Martinet. By generalizing the method of Lavaurs, the article \cite{MRR} finally gave a complete modulus for the unfolding of a germ of codimension $1$ diffeomorphism. 

The further generalization to codimension $k>1$ requires precisely the study of the polynomial vector fields on $\CP^1$ initiated by Douady and Sentenac, and this was the motivation for their study. 
A work in progress by Colin Christopher and the author on the realization for the parabolic point of codimension $2$ is using exactly the bifurcation diagram for cubic vector fields presented in this paper. 

It is striking that the geometry of the family of polynomial vector fields \eqref{pol_vf} is relevant for a large class of bifurcation problems, namely the bifurcations in the unfoldings of several resonant singularities of codimension $k$, which all have the property that the change of coordinates to normal form is divergent but $k$-summable: such singularities include the codimension $k$ parabolic fixed point of a 1-dimensional complex diffeomorphism, the resonant saddle point in $\C^2$ (and hence the Hopf bifurcation), the saddle-node (studied in \cite{RT}), and the non resonant irregular singular point of Poincar\'e rank $k$ (studied in \cite{HLR}). The works \cite{RT} and \cite{HLR} give analytic moduli of classification for the analytic unfoldings of the codimension $k$  singularities, but none of them present the bifurcation diagram. 

\bigskip

The bifurcation diagram of cubic vector fields presented here is also interesting in itself. The phase portraits are organized by the pole at infinity: in its neighborhood,  the trajectories are organized as for a saddle with two attracting and two repelling separatrices. The only (real) codimension $1$ bifurcations are bifurcations of homoclinic connections between the separatrices of $\infty$. The higher codimension bifurcations are bifurcations of multiple singular points, and simultaneous bifurcations of lower codimension. 
The parameter space is $4$-dimensional and the bifurcation diagram has a conic structure, allowing to describe it for parameter values in $\S^3$.  The main difficulty of the study is understanding the nontrivial spatial organization of the bifurcation surfaces in parameter space. 
The highest codimension bifurcations are the organizing centers of the bifurcation diagram.
\bigskip

It  is also natural to consider \eqref{pol_vf} as an ODE with complex time. If $z_1, \dots, z_{k+1}$ are the singular points, then $\Omega= \CP^1\setminus\{z_1, \dots, z_{k+1}\}$, is the orbit of a single point. But one has to pay attention that the time is ramified when $k>1$. For generic vector fields, Douady and Sentenac exploited this idea by covering $\Omega$ with sectors corresponding to strips in the time variable $t\in \C$. In the case $k=2$, we show how to generalize their description when the vector field has a homoclinic connection.

\section{The study of cubic polynomial vector fields }\label{sec:pol} 

In this section, we study the different bifurcations
of the real trajectories of the vector field
\begin{equation}\dot z=\frac{dz}{dt}=P_\eps(z)= z^3+\eps_1z+\eps_0, \qquad z\in \CP^1, \qquad \eps=(\eps_1,\eps_0)\in \C^2.\label{vf_3}\end{equation}

\subsection{Preliminaries on polynomial vector fields in $\C$}

\begin{theorem} {\rm \cite{DS}} \begin{enumerate}
\item A polynomial vector field \eqref{pol_vf} has a pole at infinity of order $k-1$ and $2k$ separatrices (see Figure~\ref{eps_zero}(a))
\item A simple  singular point in the finite plane is a strong focus, a center, or a radial node. A sufficient condition for a singular point $z_j$ to be a center is that its eigenvalue belongs to $i\R^*$. In that case, the center is isochronous and its period (given by the residue theorem) is $\frac1{2\pi| i P_\eps'(z_j)|}$.
\item A polynomial vector field \eqref{pol_vf} has no limit cycle.
\item Whenever the singular points split into two groups $\{z_j\, :\, j\in I_1\}$ and $\{z_j\, :\, j\in I_2\}$, such that $\sum_{j\in I_1}\frac1{P_\eps'(z_j)}\in i\R^*$ (and hence $\sum_{j\in I_2}\frac1{p_\eps'(z_j)}\in i\R^*$, since $\sum_{j=1}^{k+1}\frac1{p_\eps'(z_j)}=0$ on $\CP^1$), then there exists a homoclinic connection between two separatrices of $\infty$, which separates the two groups of singular points. 
 In the particular case $\#I_\ell=1$, then the corresponding singular point is a center. 
\item A separatrix of $\infty$ either ends in a finite singular point, or makes a homoclinic loop by merging with a second separatrix of $\infty$. 
\end{enumerate} \end{theorem}

\subsection{The vector field at $\infty$}

When $\eps=0$, the vector field \eqref{vf_3} has 
a triple singular point at the origin and a pole of order 1 at infinity.  The pole has two attracting and two repelling separatrices converging to the origin (see Figure~\ref{eps_zero}).
\begin{figure}\begin{center}
\subfigure[A neighborhood of $\infty$] {\includegraphics[width=4.5cm]{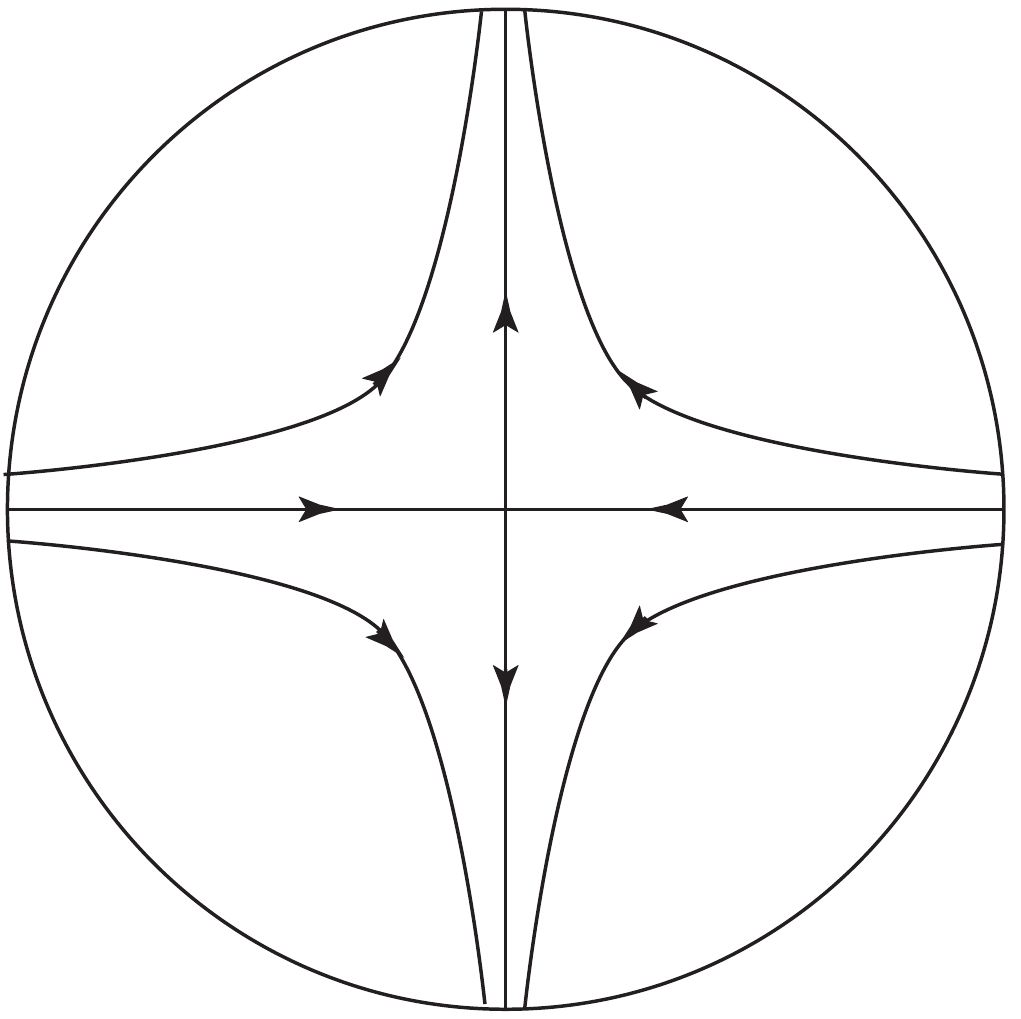}}\qquad \qquad \subfigure[A disk in the finite plane for $\eps=0$]{\includegraphics[width=4.5cm]{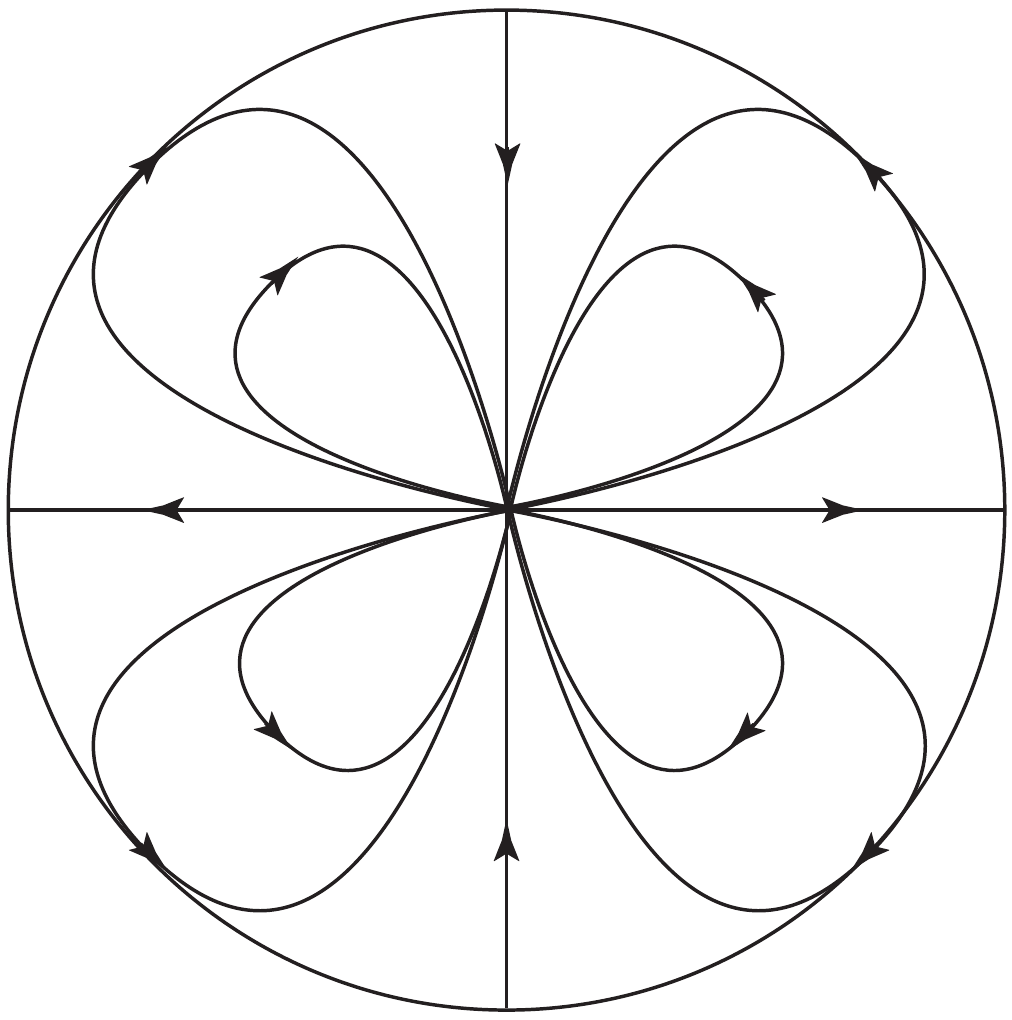}}
\caption{To give the phase portrait on $\CP^1$, it suffices to give it on two disks: one around $\infty$ and one around $0$. (a) gives the phase portrait near $\infty$ for all $\eps$, and (b) near the origin for $\eps=0$. Note that these phase portraits  naturally define four quadrants in $\CP^1$.} \label{eps_zero}\end{center}\end{figure}

For any $\eps$, the pole at infinity and its separatrices organize the phase space as noted by Douady and Sentenac in \cite{DS}. 
\smallskip

The full phase portrait on $\CP^1$ is obtained by gluing the phase portrait near $\infty$ given in Figure~\ref{eps_zero}(a) with the phase portrait on a disk in the finite plane. 

\subsection{The conic structure}

The  change $(z,t)\mapsto (Z,T)=\left(\delta z,\frac{t}{\delta^2}\right)$  transforms \eqref{vf_3} into $\frac{dZ}{dT} = Z^3+ \eps_1\delta^2Z+ \eps_0\delta^3$. This induces an equivalence relation over the parameter space
\begin{equation}\eps\simeq\eps' \Longleftrightarrow \exists \delta >0 \: (\eps_1',\eps_0')= (\delta^2\eps_1,\delta^3\eps_0).\label{simeq}\end{equation}
Hence, it is sufficient to describe the intersection of the bifurcation diagram 
with a $3$-dimensional real sphere $\S^3=\{\eps\; : \; \pl\eps\pl=\mathrm{Cst}\}$, or use charts of the form $|\eps_j|=\mathrm{Cst}$. All strata will be cones on this bifurcation diagram and will be adherent to $\eps=0$.  In particular, we can
always suppose that $\eps\neq0$. Of course any $\mathrm{Cst}$ can be taken. Depending on the context, we may choose different values so as to simplify the computations.

Note that $\S^3$ is toplogically equivalent to the completion of $\R^3$ with a point at infinity, which we will denote $\ov{\R^3}$.

\subsection{Bifurcations of real codimension $1$} 
Because of the special form of the system, all bifurcations of multiple singular points are of complex codimension $1$ or $2$, hence of real codimension $2$ or $4$. 

Therefore, the only bifurcations of real codimension $1$ are the bifurcations of homoclinic loop when two separatrices of
the pole at infinity coalesce. There are four types of connections depending
on the quadrant where they occur: we note them $(H_j)$, $j=1, \dots,
4$: see Figures~\ref{eps_zero}(a) and \ref{4_hom}. 
\begin{figure}\begin{center}
\subfigure[$(H_1)$] {\includegraphics[width=3cm]{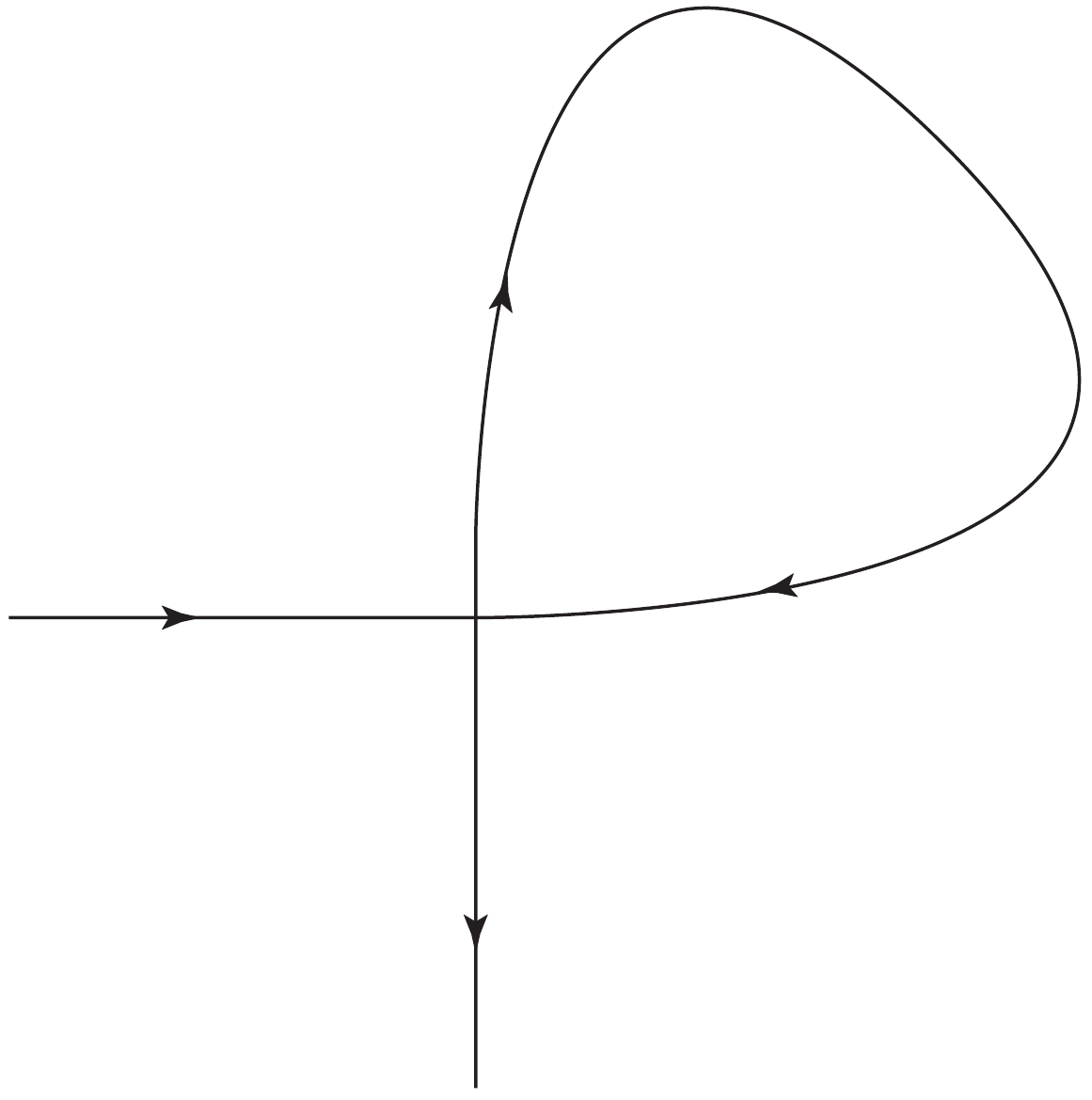}}\qquad\subfigure[$(H_2)$] {\includegraphics[width=3cm]{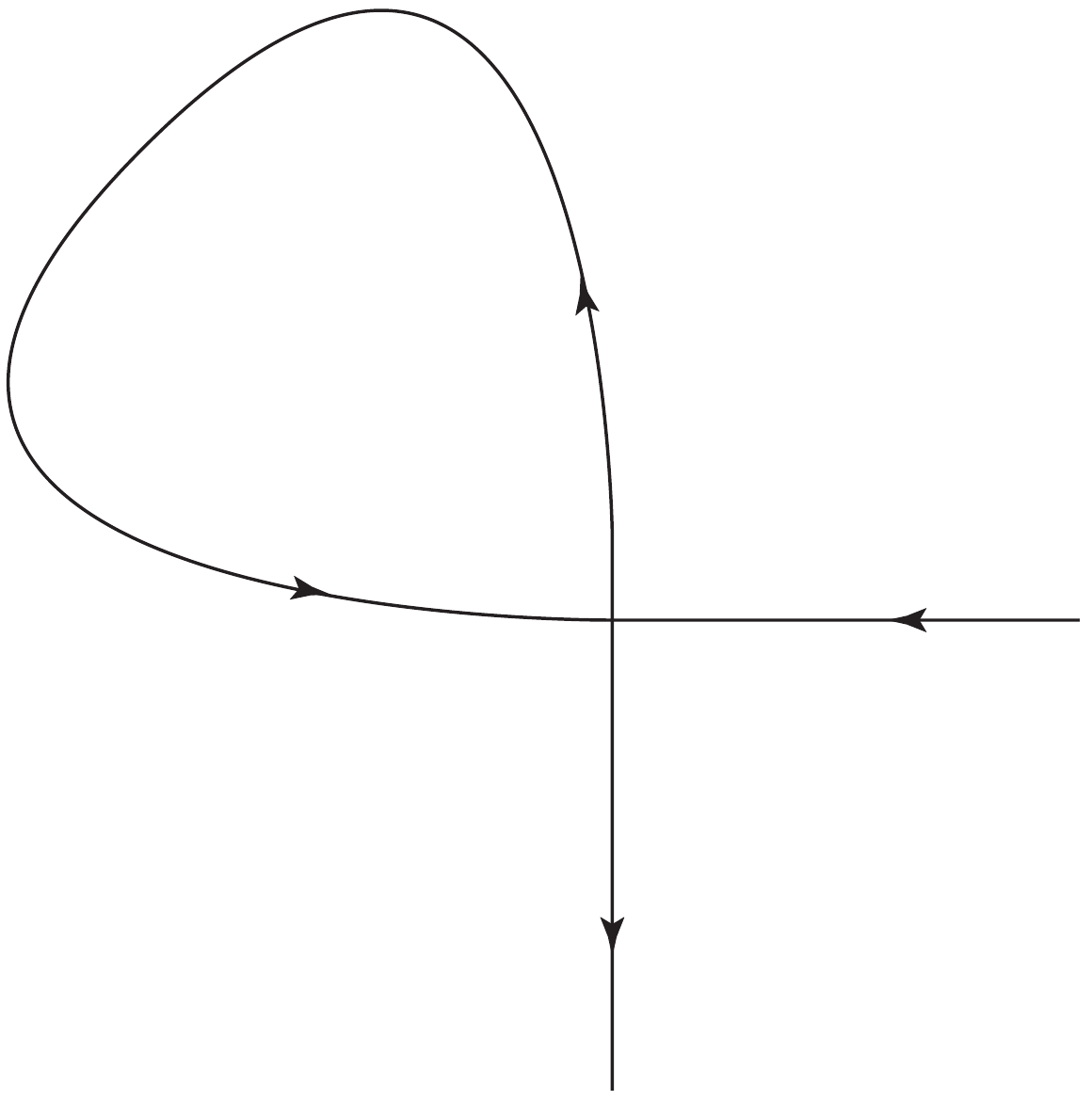}}\qquad\subfigure[$(H_3)$] {\includegraphics[width=3cm]{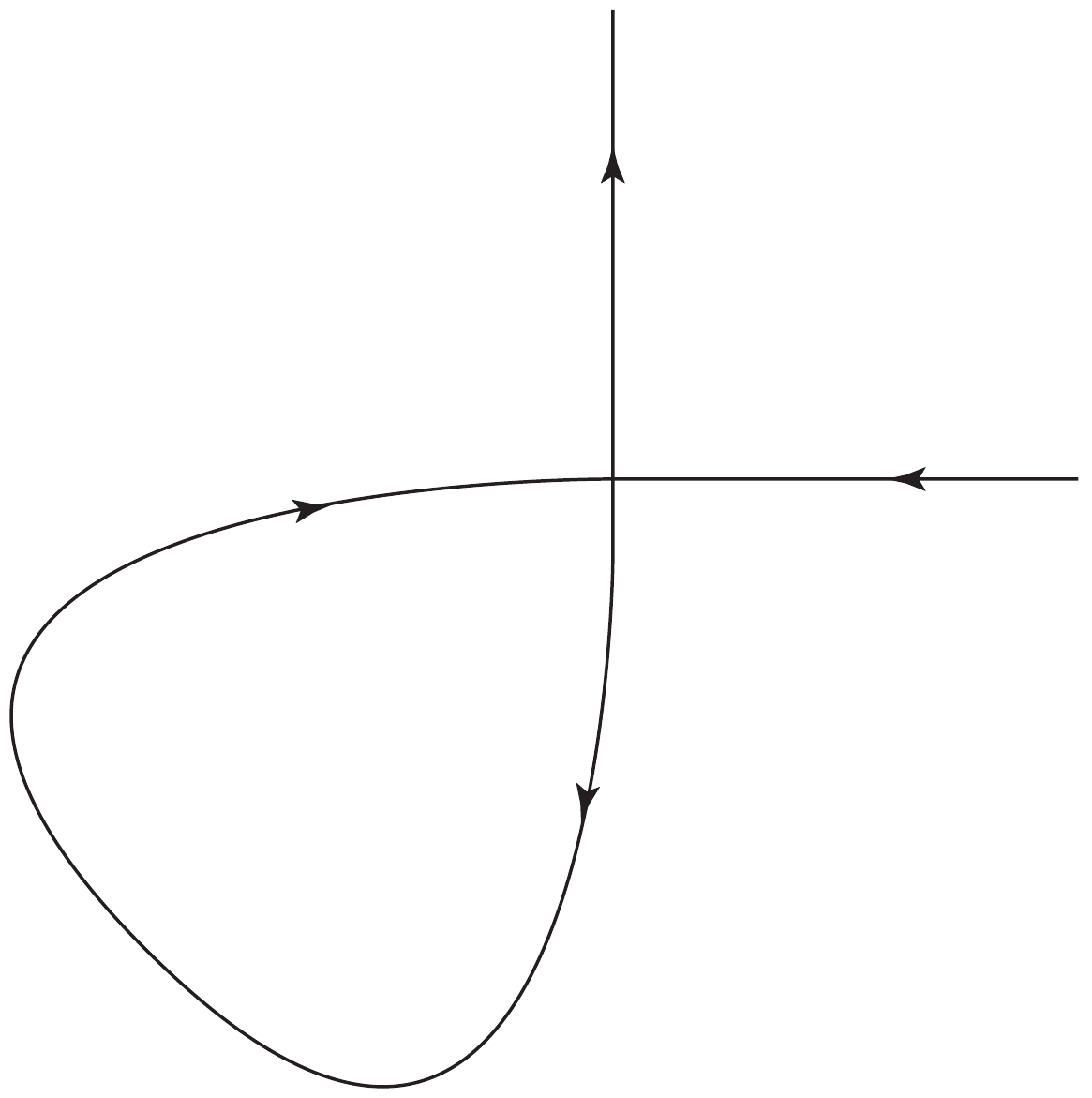}}\qquad\subfigure[$(H_4)$] {\includegraphics[width=3cm]{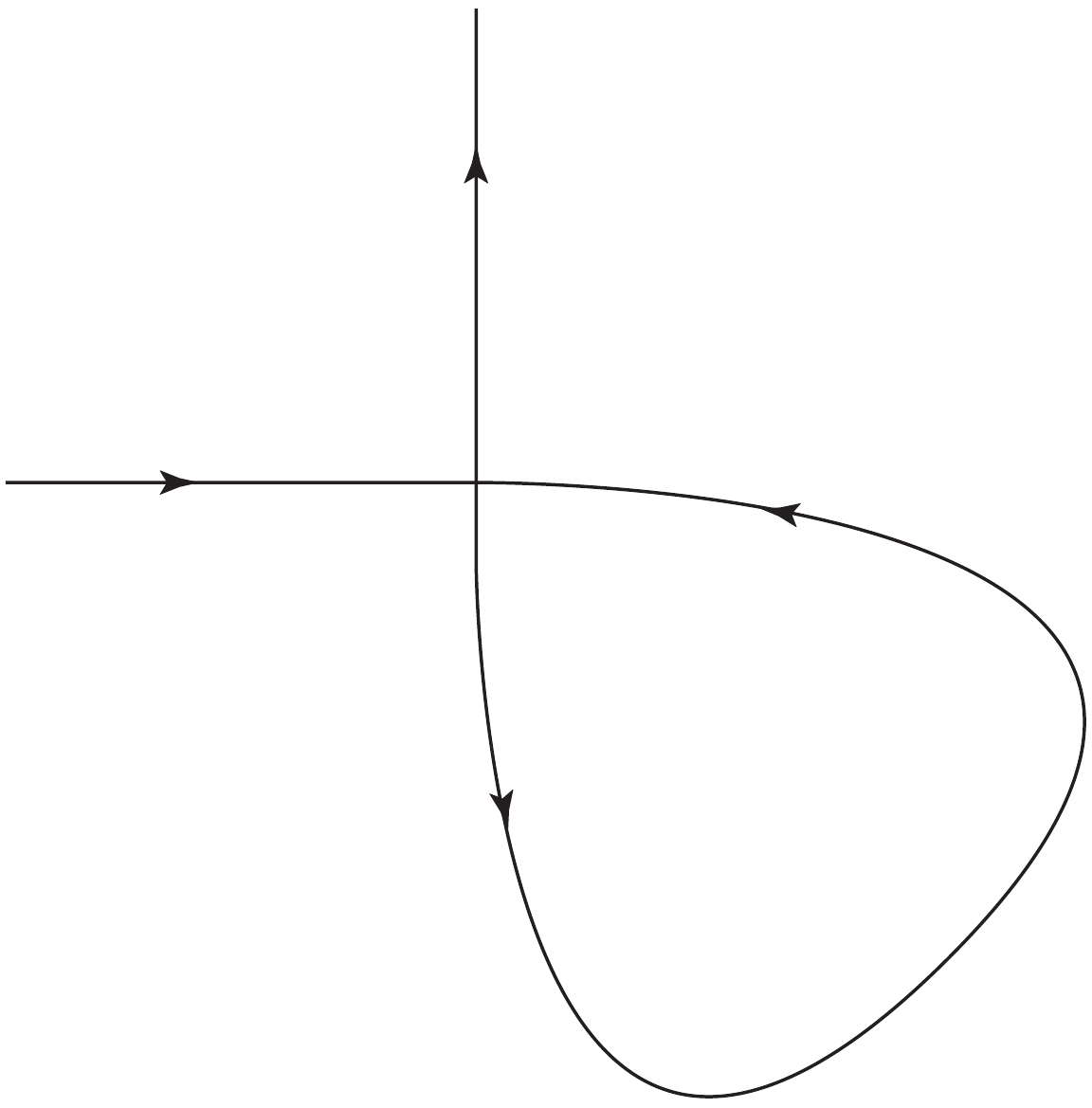}}
\caption{The four homoclinic loops in the four quadrants in $\CP^1$ inside a disk around $\infty$.} \label{4_hom}\end{center}\end{figure}

Each time a homoclinic loop occurs, it surrounds a singular
point with a pure imaginary eigenvalue. Each connection corresponds
to a bifurcation of real codimension $1$. Because there are four
separatrices, there can be up to two simultaneous homoclinic loops  forming a figure-eight loop: this will be a bifurcation of real codimension $2$.
Considering that the system is invariant under
\begin{equation}(z,\eps_1,\eps_0,t)\mapsto(iz,-\eps_1, -i\eps_0, -t),\label{rotation4}\end{equation} the
study of one  bifurcation surface $(H_j)$ allows determining  the
others.
These bifurcations are represented by surfaces in $\S^3$. The equation of the bifurcation locus is implicit, and hence not easy to visualize. It is of the form $P_\eps'(z_j)\in i\R^*$, with boundary points when $P_\eps'(z_j)=0$.

\subsection{Bifurcation of parabolic points.} This occurs when two
singular points coallesce in a double singular point (parabolic point), namely when the discriminant $\Delta$
vanishes, where
$$\Delta= -4\eps_1^3-27\eps_0^2.$$ Note that this local bifurcation is of
complex codimension $1$, hence real codimension $2$, but it is represented by a curve on $\S^3$ because of the conic structure. 

A generic bifurcation of double singular point  occurs when the third singular point is not a center, i.e. its eigenvalue is not on the imaginary axis.

\begin{figure}\begin{center}
\subfigure[Codimension 2] {\includegraphics[width=5cm]{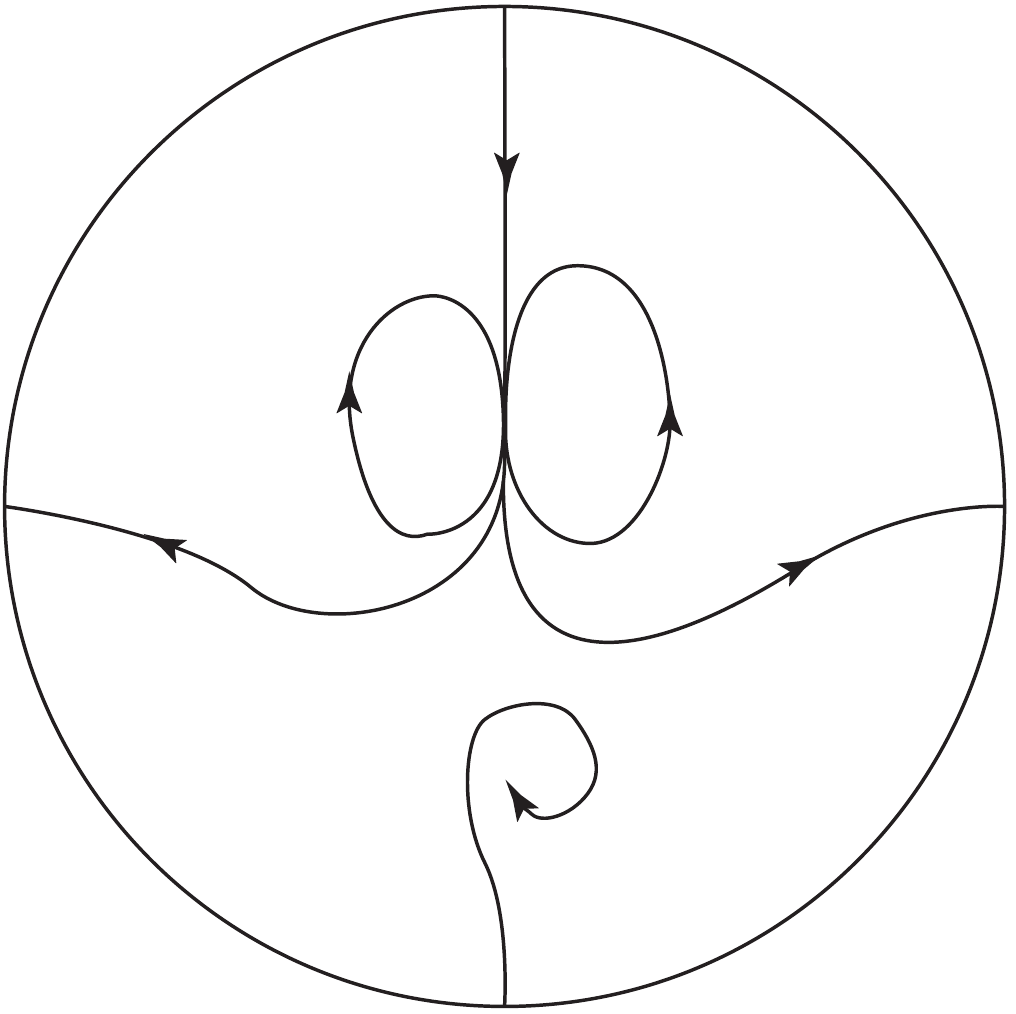}}\qquad\qquad\subfigure[Codimension 3] {\includegraphics[width=5cm]{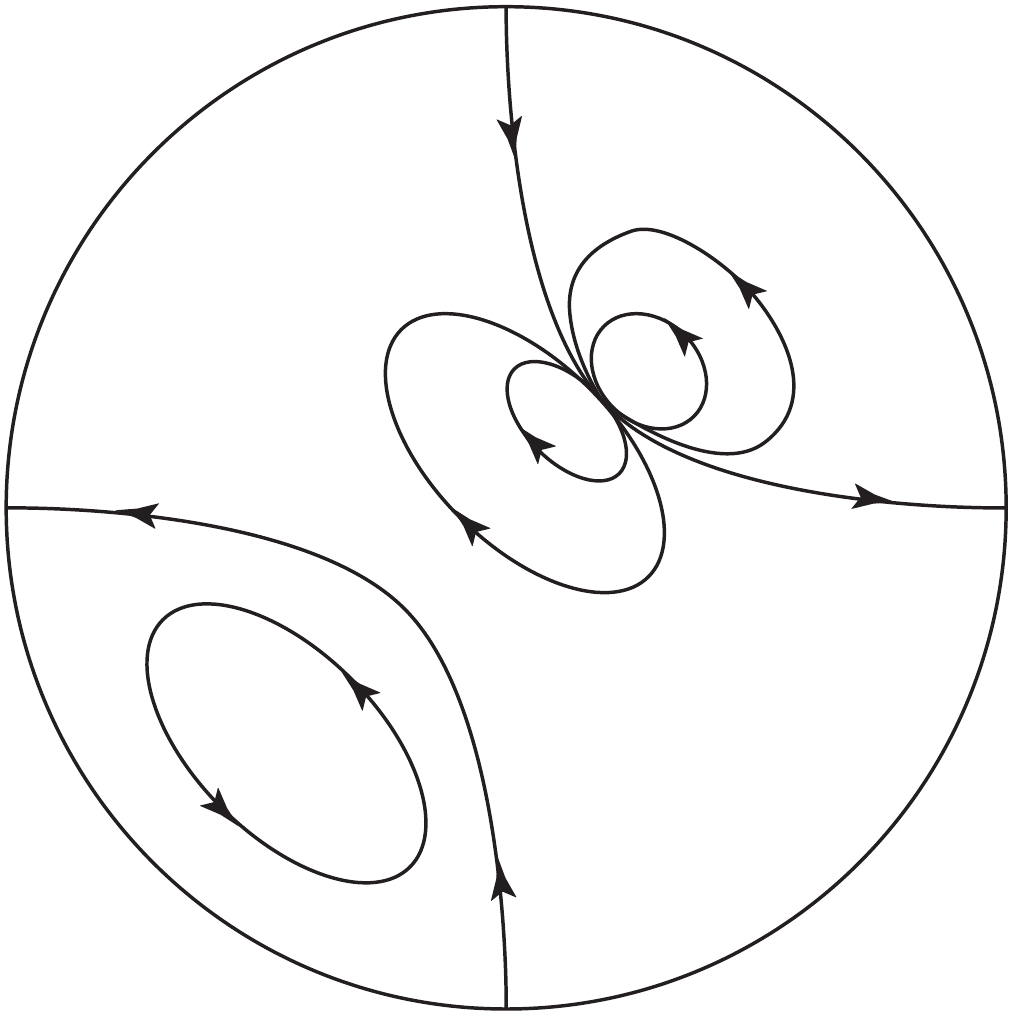}}
\caption{The two situations for parabolic point: (a) corresponds to a regular point of $\Delta=0$, with two heteroclinic connections through the parabolic point (here $(H_1)$ and $(H_2)$), and (b) to a point of codimension 3 with one heteroclinic connection through the parabolic point  (here $(H_1)$), and an additional homoclinic connection (here $(H_3)$).} \label{fig_parabolic}\end{center}\end{figure}

\begin{proposition}  
\begin{enumerate} \item In $\S^3$, $\Delta=0$ can be represented as a $2:3$ torus knot on the torus $4|\eps_1|^3=27|\eps_0|^2$ (i.e. a close curve turning twice (resp. thrice) around $0$ in $\eps_0$ (resp. $\eps_1$)). Four points of this knot are not generic, i.e. correspond to higher order bifurcations. These four points divide $\Delta=0$ in four arcs of regular (generic) points.
\item The bifurcation diagram at a regular point of $\Delta=0$ (when the third singular point is not a center) is given in Figure~\ref{Delta_regular}.
In particular, two surfaces of homoclinic bifurcations in adjacent quadrants end along the regular arcs of $\Delta=0$ (see Figure~\ref{fig_parabolic}(a)).
\end{enumerate}\end{proposition}
\begin{figure}\begin{center}
\includegraphics[width=8 cm]{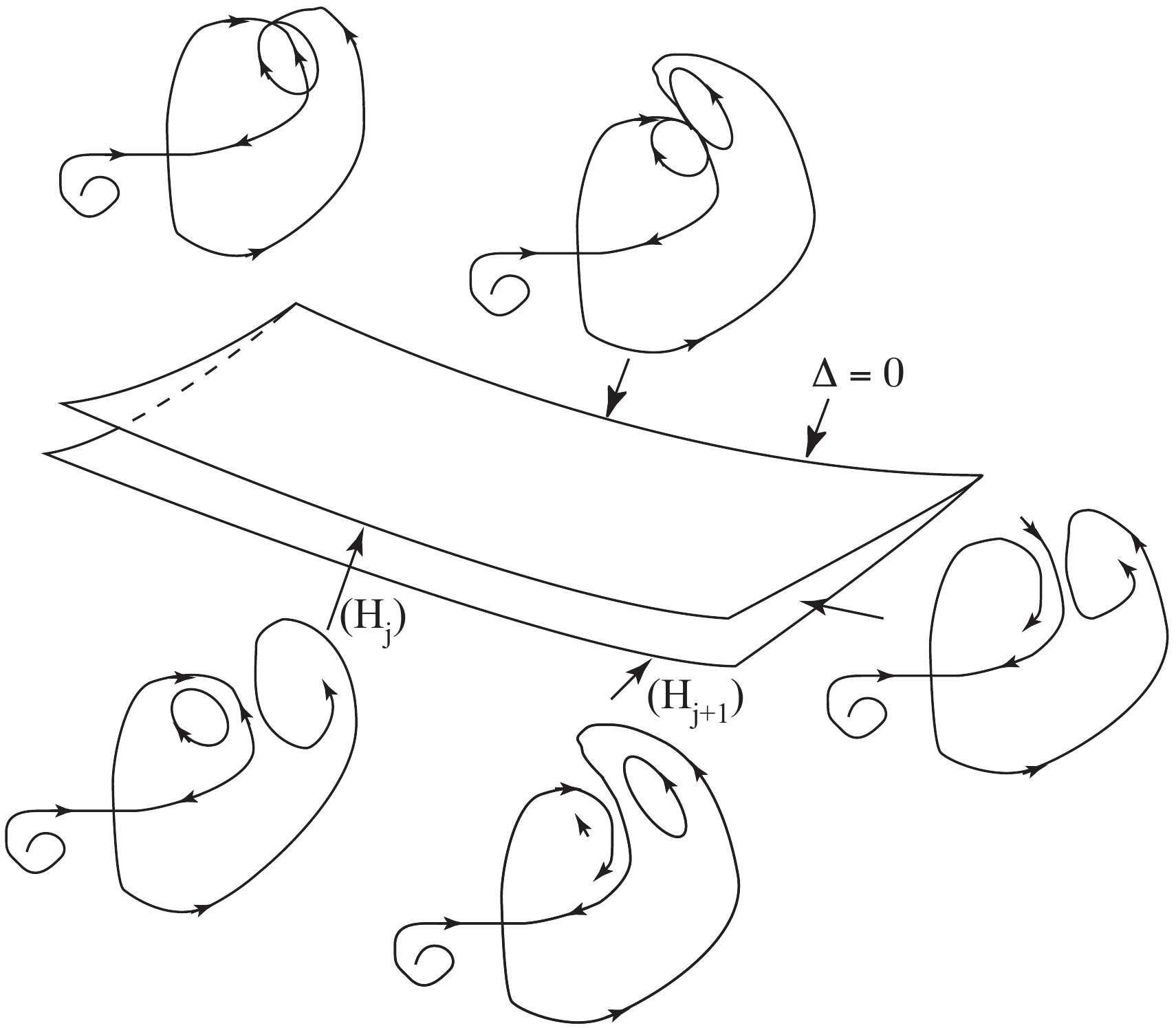} \caption{Bifurcation diagram near a generic point of $\Delta=0$.}
  \label{Delta_regular}\end{center}
\end{figure}
\begin{proof}
When  $\Delta=0$, then $|\eps_0|= \sqrt{\frac{4|\eps_1|^3}{27}}$. Hence, instead of cutting the bifurcation diagram by a sphere $\pl \eps\pl =\mathrm{Cst}$,  we prefer to use a chart $|\eps_1|=1$ (i.e. $\eps_1\in \S^1$). At a
point of $\Delta=0$ corresponding to a particular $\eps'$,
$P_{\eps'}(z)$ can be written as
$$P_{\eps'}(z)=(z-a)^2(z+2a),$$
and a general unfolding for $\eps$ close to $\eps'$ has the form
$$P_\eps(z)=((z-a-\eta_1)^2-\eta_2)(z+2(a+\eta_1)).$$
It is easily checked that the change $(\eps_1,\eps_0)\mapsto
(\eta_1,\eta_2)$ is an analytic diffeomorphism in the neighborhood
of $\eps'$ when $a\neq0$. Note that $P_{\eps'}'(-2a)=9a^2$. Hence,
$P_{\eps'}'(-2a)\in i\R$ if and only if $a\in (1\pm i)\R$. Since
$\eps_1=-3a^2$, when we suppose $|\eps_1|=1$, this yields four
points on the torus knot: $a=\pm\frac{1\pm i}{\sqrt{3}}$, symmetric to each other under \eqref{rotation4}. These four points in parameter space have codimension
$3$, since the third singular point is a center surrounded by a homoclinic loop. They cut $\Delta=0$ into four regions, which are sent one to
the other under \eqref{rotation4}. At points of $\Delta=0$, the
homoclinic connections become heteroclinic connections through the parabolic point in phase space (Figure~\ref{fig_parabolic}). At a
regular point  of $\Delta$,
there are two homoclinic connections in adjacent
sectors (Figure~\ref{fig_parabolic}(a)). 
 Near such a point, two half-surfaces $(H_j)$ and $(H_{j+1})$ merge on $\Delta=0$
(indices are $(\text{mod} \; 4)$). These half-surfaces are tangent
on $\Delta=0$. Indeed, they correspond to 
$$P_\eps'(a+\eta_1\pm \sqrt{\eta_2}) = \pm2\sqrt{\eta_2}\left(3(a+\eta_1)\pm \sqrt{\eta_2}\right)\in i\R.$$
Let us put $\sqrt{\eta_2}=x+iy$ and let $3(a+\eta_1)=\alpha+i\beta$. Since, by hypothesis, $P_\eps'(-2a)\notin i\R$ for $\eta_1=\eta_2=0$, then $\alpha\pm \beta\neq0$ for small $\eta_1,\eta_2$.
We have that $\mathrm{Re}(P_\eps'(a+\eta_1\pm \sqrt{\eta_2}) )=0$ if and only if $\pm(x^2-y^2)+\alpha x-\beta y=0$. Because $\alpha\neq\pm\beta$, these two hyperbolas are distinct and have a quadratic tangency at the origin in the $(x,y)$-space, i.e. in the $\sqrt{\eta_2}$-space. Then the corresponding curves in $\eta_2$-space are also tangent. A finer analysis would show that the tangency if of order $\frac32$.
\end{proof}

\subsection{The other bifurcations}

Apart from the codimension $4$ bifurcation at $\eps=0$, there remains two bifurcations, one  of codimension $2$, and one of codimension $3$. The codimension $2$ bifurcation is the figure-eight loop,  i.e. the intersection of $(H_j)$ and
$(H_{j+2})$. It occurs when two (and hence
three) eigenvalues are pure imaginary, i.e. all singular points are centers. 

\begin{proposition} The intersection of the codimension $2$ bifurcation $(H_j)\cap (H_{j+2})$ (denoted $(H_{j,j+2})$) with   $|\eps_1|=1$, occurs along the two line
segments $$J_\pm=\left\{(\eps_1,\eps_0)\; : \;\eps_1=\pm i, \eps_0\in(1\pm
i)\left[-\sqrt{\frac2{27}},\sqrt{\frac2{27}}\right]\right\},$$ with
endpoints on $\Delta=0$. The bifurcation diagram is given in
Figure~\ref{figureeight}.
\begin{figure}\begin{center}
\includegraphics[
  width=10 cm]{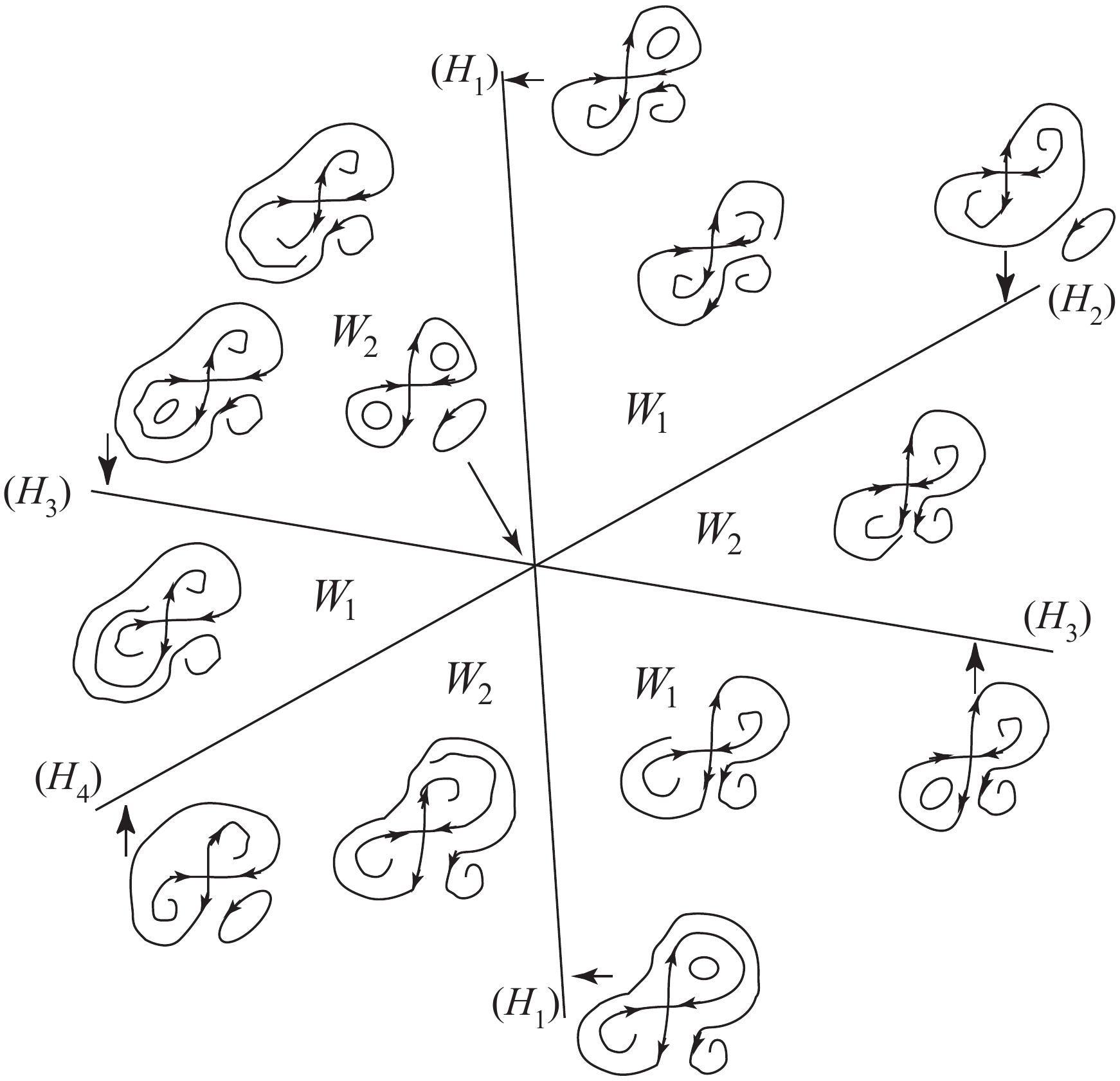} \caption{Transversal cut of the bifurcation diagram for the figure eight loop. The point of saddle type is the pole at $\infty$. The phase portraits represent the sphere minus a finite point. The two open sets $W_1$ and $W_2$ of generic values of parameters are highlighted: from this, it is not obvious to see that each $W_j$ is simply connected! }
  \label{figureeight}\end{center}
\end{figure}
\end{proposition}\begin{proof} We
can suppose that for some $\eps'$
$$P_{\eps'}(z) = (z-a)(z+b)(z+a-b). $$ Hence $P_{\eps'}'(a)=
(a+b)(2a-b)$, $P_{\eps'}'(-b)= -(a+b)(a-2b)$, and $P_{\eps'}'(b-a)=
(b-2a)(2b-a)$. Since the sum of the inverses of these derivatives
vanishes, as soon as two of them  belong to $i\R$, the third belongs
to $i\R$. Moreover, $4P_{\eps'}'(a)+ P_{\eps'}'(-b) + P_{\eps'}'(b-a)= 9a^2\in i\R$, and $P_{\eps'}'(a)+ 4P_{\eps'}'(-b) + P_{\eps'}'(b-a)= 9b^2\in i\R$, yielding $a,b \in (1\pm i)\R$. Using
the symmetry \eqref{rotation4}, we can suppose $a=\alpha(1+i)$,
$b=\beta(1+i)$, $\alpha,\beta>0$. It is easily checked that when the
singular points are distinct, then the unfolding for $\eps$ near
$\eps'$ has the form $$P_{\eps}(z) =
(z-a-\eta_1)(z+b+\eta_2)(z+a-b+\eta_1-\eta_2).
$$ Indeed, the change $(\eta_1,\eta_2)\mapsto (\eps_1,\eps_0)$ is
invertible when  $(2a-b)(a-2b)(a+b)\neq0$.

There are surfaces of homoclinic loop bifurcations corresponding to
the three conditions
\begin{itemize}
\item $P_\eps'(a+\eta_1)\in i\R$;
\item $P_\eps'(-b-\eta_2)\in i\R$;
\item $P_\eps'(b-a-\eta_1+\eta_2)\in i\R$.
\end{itemize}
We want to show that these surfaces are transversal when the
singular points are distinct. For this, we let
$$\begin{cases}
\eta_1=\delta_1(1+i)+\nu_1(1-i),\\
\eta_2=\delta_2(1+i)+\nu_2(1-i).
\end{cases}$$
Then \begin{align*} {\rm
Re}(P_\eps'(a+\eta_1))&=2[\nu_1(4\alpha+\beta)+
\nu_2(\alpha-2\beta)]+O(|(\eta_1,\eta_2)|^2),\\
{\rm Re}(P_\eps'(-b-\eta_2))&=2[\nu_1(-2\alpha+\beta)+
\nu_2(\alpha+4\beta)]+O(|(\eta_1,\eta_2)|^2),\\
{\rm Re}(P_\eps'(-a+b-\eta_1+\eta_2))&=2[\nu_1(4\alpha-5\beta)+
\nu_2(-5\alpha+4\beta)]+O(|(\eta_1,\eta_2)|^2),
\end{align*}
Two of these surfaces are trasnversal if the corresponding $2\times 2$ Jacobian with respect to  $(\nu_1,\nu_2)$ does not vanish. The three surfaces are two by two transversal when
\begin{equation}(2a-b)(a-2b)(a+b)\neq0.\label{transversal} \end{equation}\end{proof}

The last bifurcation occurs for non regular points of $\Delta=0$. At these points, the system has a parabolic point and a center: see Figure~\ref{fig_parabolic}(b). 

\begin{proposition}
The bifurcation diagram for a codimension $3$ point as in Figure~\ref{fig_parabolic}(b) is given in
Figure~\ref{cod3}.
\begin{figure}\begin{center}
\includegraphics[width=10 cm]{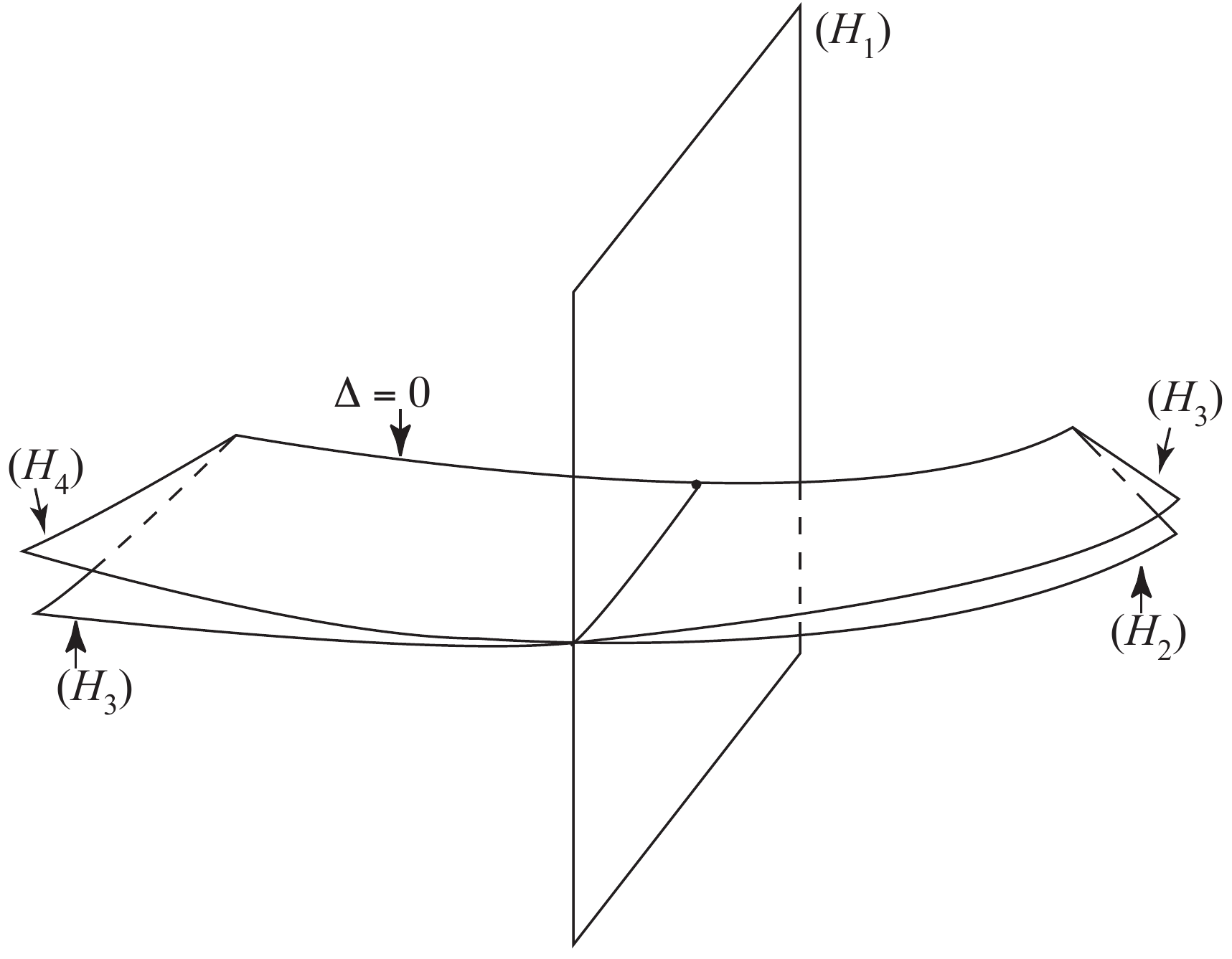} \caption{Bifurcation diagram for a codimension $3$ point.}
  \label{cod3}\end{center}
\end{figure}
\end{proposition}
\begin{proof} This situation corresponds to a non regular point of $\Delta=0$, which is an endpoint of  a segment of double homoclinic curve (figure-eight loop), since it can be unfolded into an $(H_{j,j+2})$. 
The only bifurcation surfaces are the four homoclinic loop surfaces which all merge along the two bifurcations $(H_{j,j+2})$, $j=1,2$. In the limit, some become tangent because the transversality condition \eqref{transversal} is violated at the limit. 
\end{proof}

\section{The global bifurcation diagram}

\subsection{The two open sets $W_1$ and $W_2$ of generic values}\label{sect:open_sets} Before putting together the bifurcation diagram, we must describe the generic situation, which was studied in great detail by Douady and Sentenac \cite{DS}. When $\eps$ is not a bifurcation value, the separatrices of $\infty$ land at singular points.  There are two generic ways in which this can occur: 
\begin{itemize}
\item the two repelling  separatrices land at the same attracting singular point, and the  two other separatrices each land  at a different repelling point: let us call $W_1$ the open set of these parameter values;
\item the two attracting  separatrices land at the same repelling singular point, and the  two other separatrices each land at a different attracting point: let us call $W_2$ this open set.
\end{itemize}
$W_1$ and $W_2$ are  simply connected sets in  $\C^2$. This is a highly nontrivial result, which would be very difficult to prove by \lq\lq classical\rq\rq\ methods, including visualizing the two domains in $\S^3$. Indeed, these domains are limited by the four surfaces $(H_i)$ of homoclinic connections, and the boundary of these surfaces is the union of the curve $\Delta=0$ with the two segments $(H_{i,i+2})$ where double homoclinic connections occur. This means that these surfaces are folded in $\S^3$.  Fortunately, Douady and Sentenac produced a very clever argument allowing to show that $W_1$ and $W_2$ are  simply connected. Indeed, for each $\eps\in W_j$, we can compute the \lq\lq complex time\rq\rq\ to go from $\infty$ to $\infty$ along a loop not cutting the connecting graph (i.e. the union of the separatrices). Up to homotopy, there are exactly two such loops, providing two non real times, and we can orient the loops so that the complex times both belong to the upper half-plane $\H$ (see Figure~\ref{fig:tau}). These complex times are given (up to sign depending on the orientation of the loop) by $\frac1{2\pi i} \mathrm{Res}\left(\frac1{P_\eps}\right)(z_j)= \frac1{2\pi i P_\eps'(z_j)} $, where $z_j$ is the singular point inside the loop. In particular, the complex times depend only on the homotopy class of the loop. Let us call $1,2,3, 4,$ the four quadrants when turning in the positive direction. Each loop joins two adjacent quadrants $i,j$.  Let  $\tau_{i,j}$  be the travel time from quadrant $i$ to quadrant $j$ along a curve from $\infty$ to $\infty$ not intersecting the connecting graph. We have two times, $\tau_{i,j}$ and $\tau_{i',j'}$, corresponding to the two non homotopic loops. The very subtle  theorem of Douady and Sentenac is the following.
\begin{figure}\begin{center}
\includegraphics[width=5.5cm]{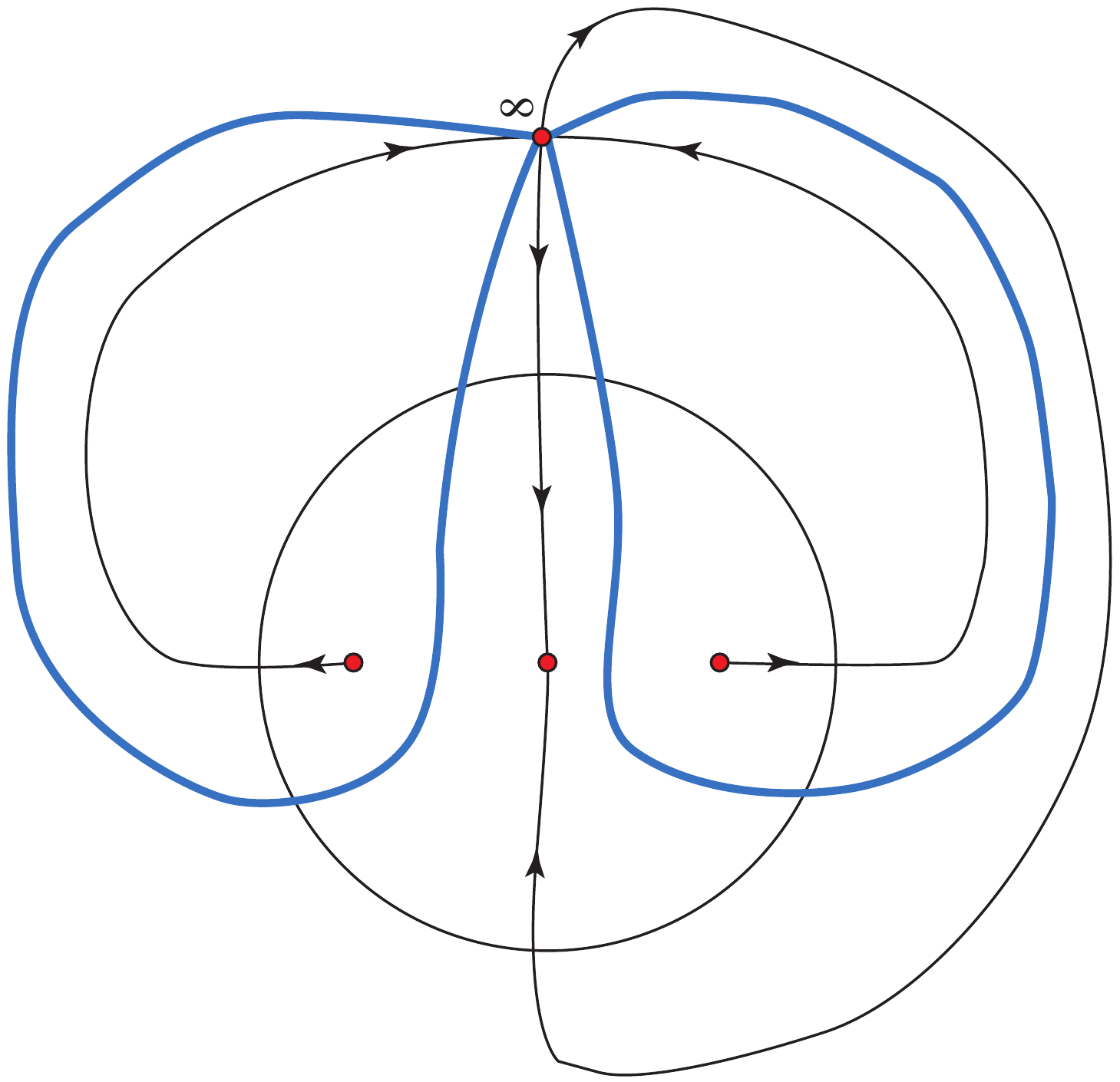}
\caption{The two loops based at infinity and providing the two times $\tau_{i,j}$ and $\tau_{i',j'}$. The figure is on $\CP^1$ minus a point. }\label{fig:tau}\end{center}\end{figure}

\begin{theorem}\label{thm:DS} \emph{\cite{DS}} The map $F: W_j\rightarrow \H^2$, given by $F(\eps)=(\tau_{i,j},\tau_{i',j'})$ is a holomorphic diffeomorphism.\end{theorem}

The tuple $(\tau_{i,j},\tau_{i',j'})$ is the \emph{Douady-Sentenac analytic invariant} announced in the introduction. We will come back to the idea of the proof in Section~\ref{sec:DS}.

\begin{corollary} $W_1$ and $W_2$ are simply connected open sets in $\C^2$. Because of the conic structure of the bifurcation diagram, their intersection with $\S^3$ is also simply connected. \end{corollary}

\subsection{The bifurcation diagram for $\eps\in \S^3$}

We now have all the ingredients to give the global bifurcation diagram for parameter values inside the sphere $\S^3$. Let us first introduce the following notation.

\begin{notation}\label{notation}

\begin{itemize}
\item $(H_{ij})$, with $j=i+2$, denotes the transversal intersection of  $(H_i)$ with $(H_j)$, whose bifurcation diagram appears in Figure~\ref{figureeight}. 
\item If we look at Figure~\ref{figureeight}, it is clear that $(H_{i,i+2})$ is in the boundary of $(H_{i\pm1})$. Looking for instance at $W_1$, then it has corners at the intersection of the closures of $(H_1)$ and $(H_2)$, $(H_1)$ and $(H_3)$, and $(H_3)$ and $(H_4)$ which we denote respectively $(H_{12})$, $(H_{13})$ and $(H_{34})$. Similarly, $W_2$ has corners at the intersection of the closures of $(H_1)$ and $(H_3)$, $(H_1)$ and $(H_4)$, and $(H_2)$ and $(H_3)$ which we denote respectively $(H_{13})$, $(H_{14})$ and $(H_{23})$.  
\item $(H_{ij}^0)$ with $j=i+1$ corresponds to a regular point of $\Delta=0$ (as in Figure~\ref{fig_parabolic}(a)), for which there are heteroclinic loops $(H_i)$ and $(H_j)$ through the parabolic point, and with bifurcation diagram as in Figure~\ref{Delta_regular}. 
\item $(H_i^0)$ corresponds to a codimension 3 point with a parabolic point, a heteroclinic loop $(H_i)$ through the parabolic point, and an additional homoclinic loop $(H_{i+2})$ (as in Figure~\ref{fig_parabolic}(b)), and with bifurcation diagram as in Figure~\ref{cod3}. 
\end{itemize} 
\end{notation}

\begin{theorem} Each of the two connected components $W_j$ of generic vector fields fills an open set of  $\S^3$. The boundary of each connected component is a union of pieces of the four homoclinic surfaces $(H_j)$, $j=1, \dots, 4$, linked along curves of codimension 2 bifurcations with vertices at codimension 3 bifurcation points. These bifurcation surfaces are organized as described in Figures~\ref{fig:pentagons} and \ref{fig:tetrahedron}. 
\begin{figure}
\begin{center}
\subfigure[$(H_1)$]{\includegraphics[width=5cm]{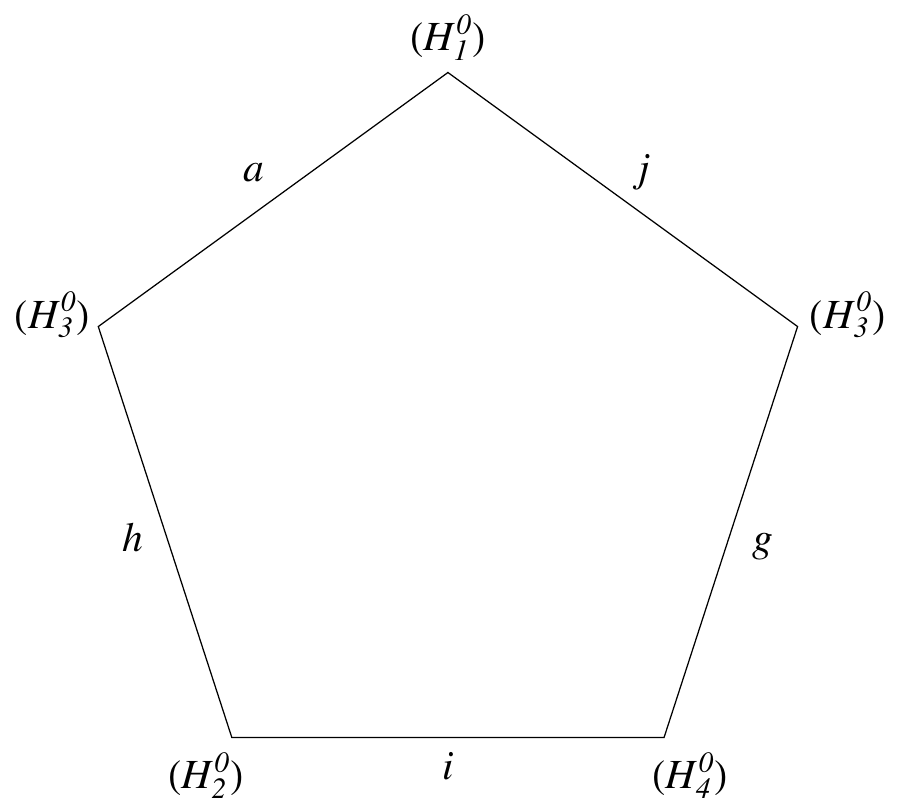}}\qquad\qquad
\subfigure[$(H_2)$]{\includegraphics[width=5cm]{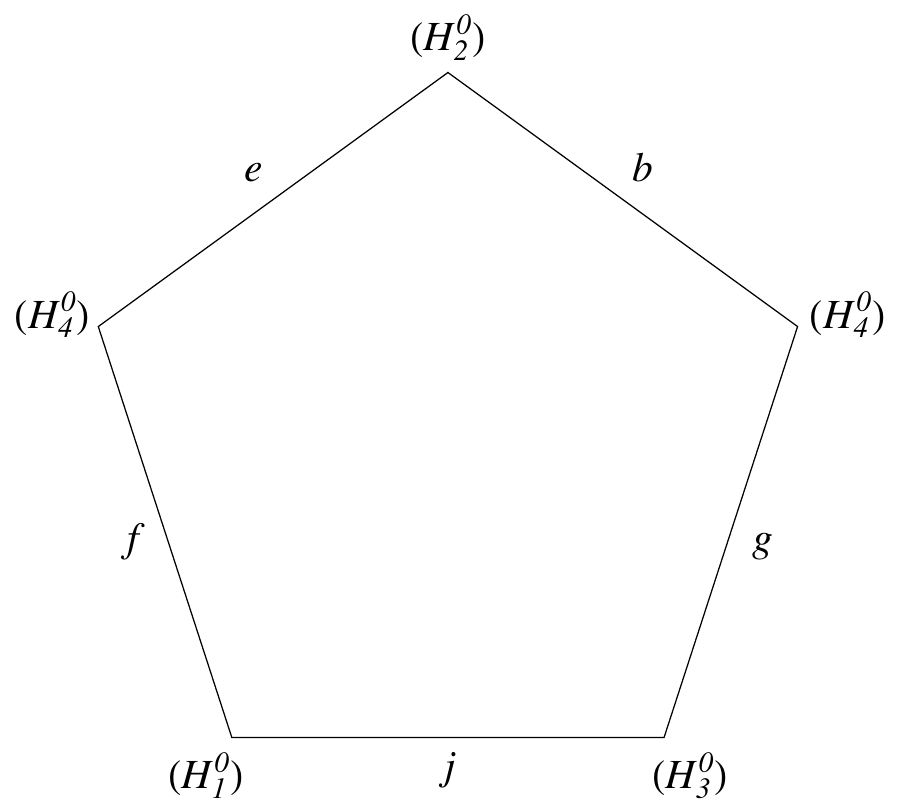}}
\subfigure[$(H_3)$]{\includegraphics[width=5cm]{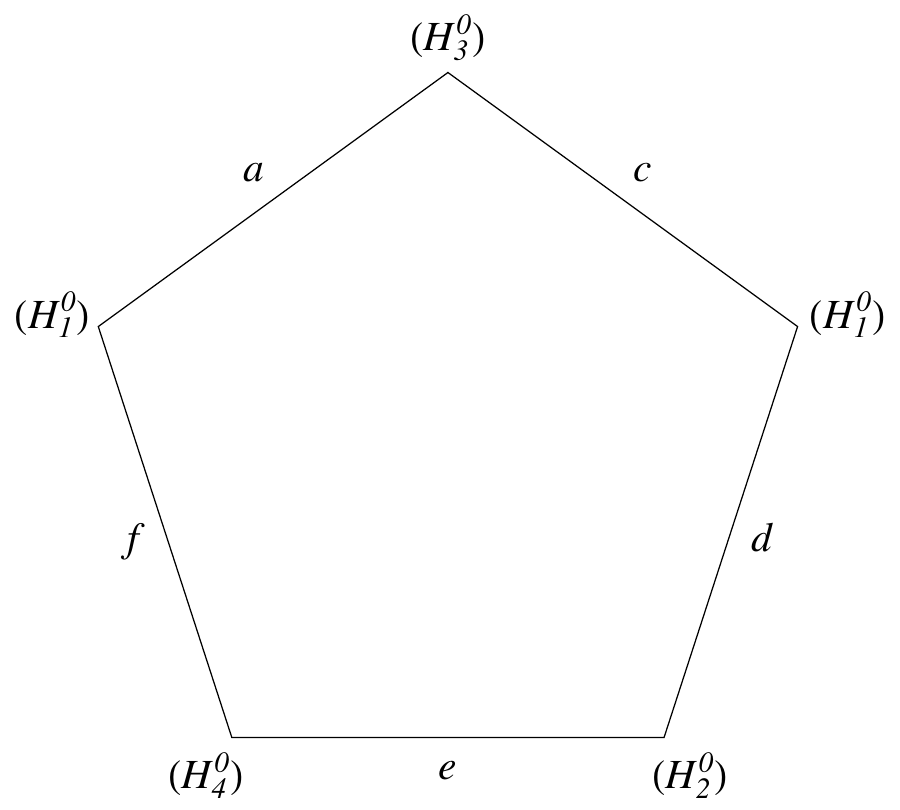}}\qquad\qquad
\subfigure[$(H_4)$]{\includegraphics[width=5cm]{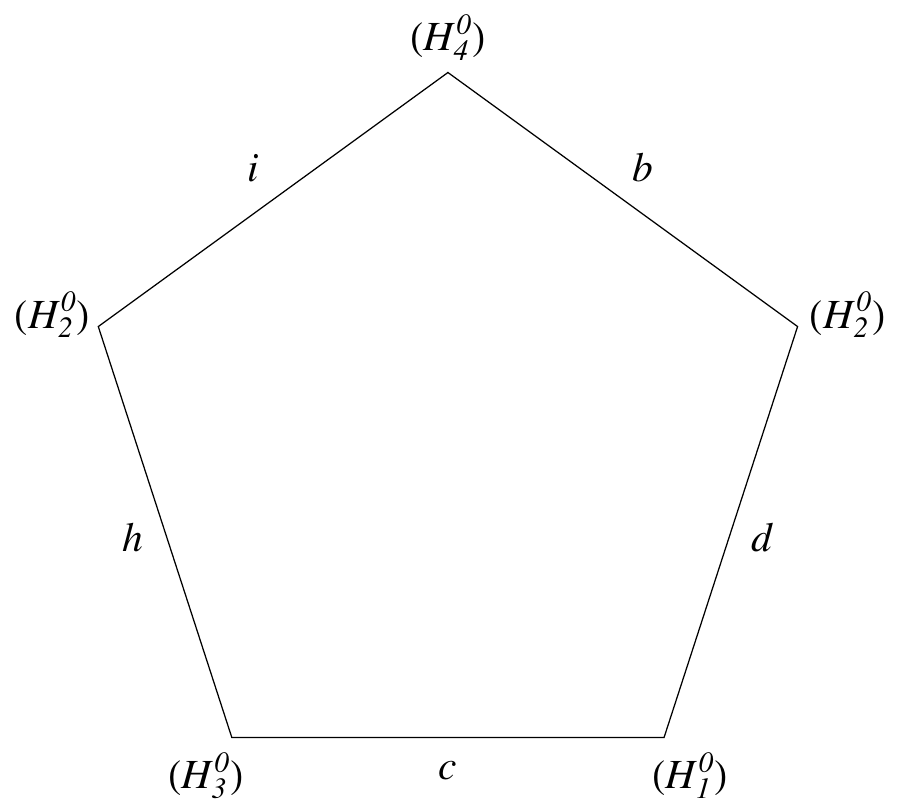}}
\caption{\label{fig:pentagons} The four surfaces of homoclinic loops and their pentagonal boundaries to be glued according to the letters marked. Notation~\ref{notation} is used. The meaning of the small letters ($a$-$j$) is the following: \lq\lq $a$\rq\rq: $(H_{13})$;  \lq\lq $b$\rq\rq: $(H_{24})$; \lq\lq $c$\rq\rq: $(H_{12})$; \lq\lq $d$\rq\rq: $(H_{12}^0) $; \lq\lq $e$\rq\rq: $(H_{14})$; \lq\lq$f$\rq\rq: $(H_{14}^0)$; \lq\lq $g$\rq\rq: $(H_{34}^0)$; \lq\lq $h$\rq\rq: $(H_{23}^0)$; \lq\lq $i$\rq\rq: $(H_{23})$; \lq\lq $j$\rq\rq: $(H_{34})$.}
\end{center}
\end{figure}
\begin{figure}
\begin{center}
\includegraphics[width=7cm]{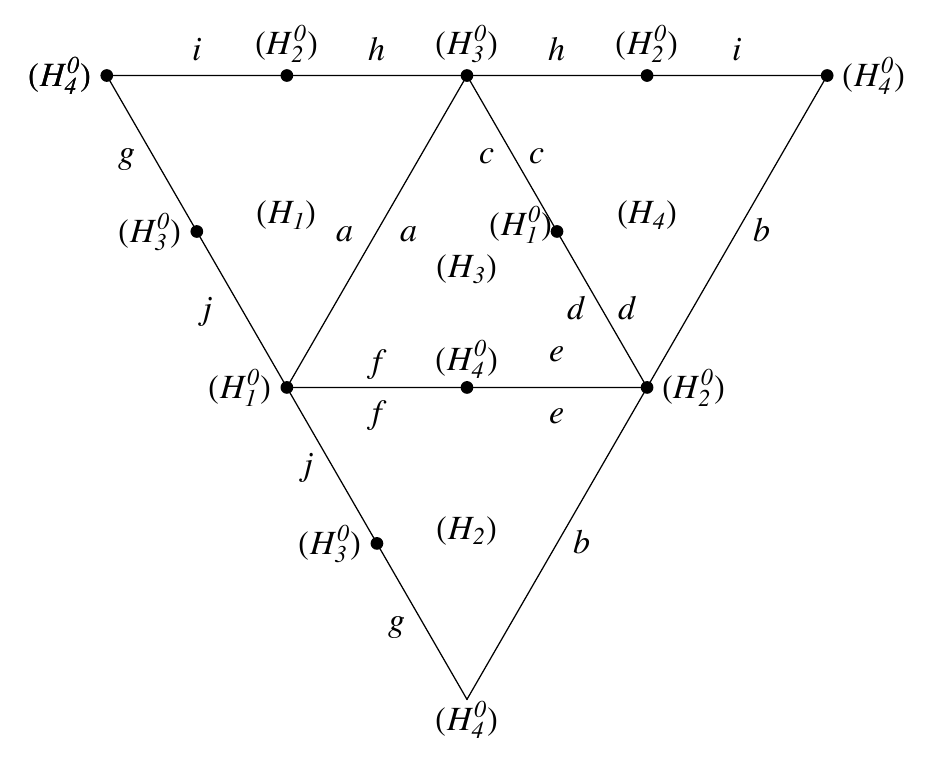}\caption{The four surfaces of homoclinic loops and their boundaries to be glued according to the letters marked. }\label{fig:tetrahedron} 
\end{center}
\end{figure}
\end{theorem}
\begin{proof} To understand Figure~\ref{fig:pentagons}, we first consider Figure~\ref{fig:double} gluing the bifurcation diagrams of the two codimension $3$ bifurcations $(H_1^0)$ and $(H_3^0)$. 
\begin{figure}\begin{center}
\includegraphics[ width=10 cm]{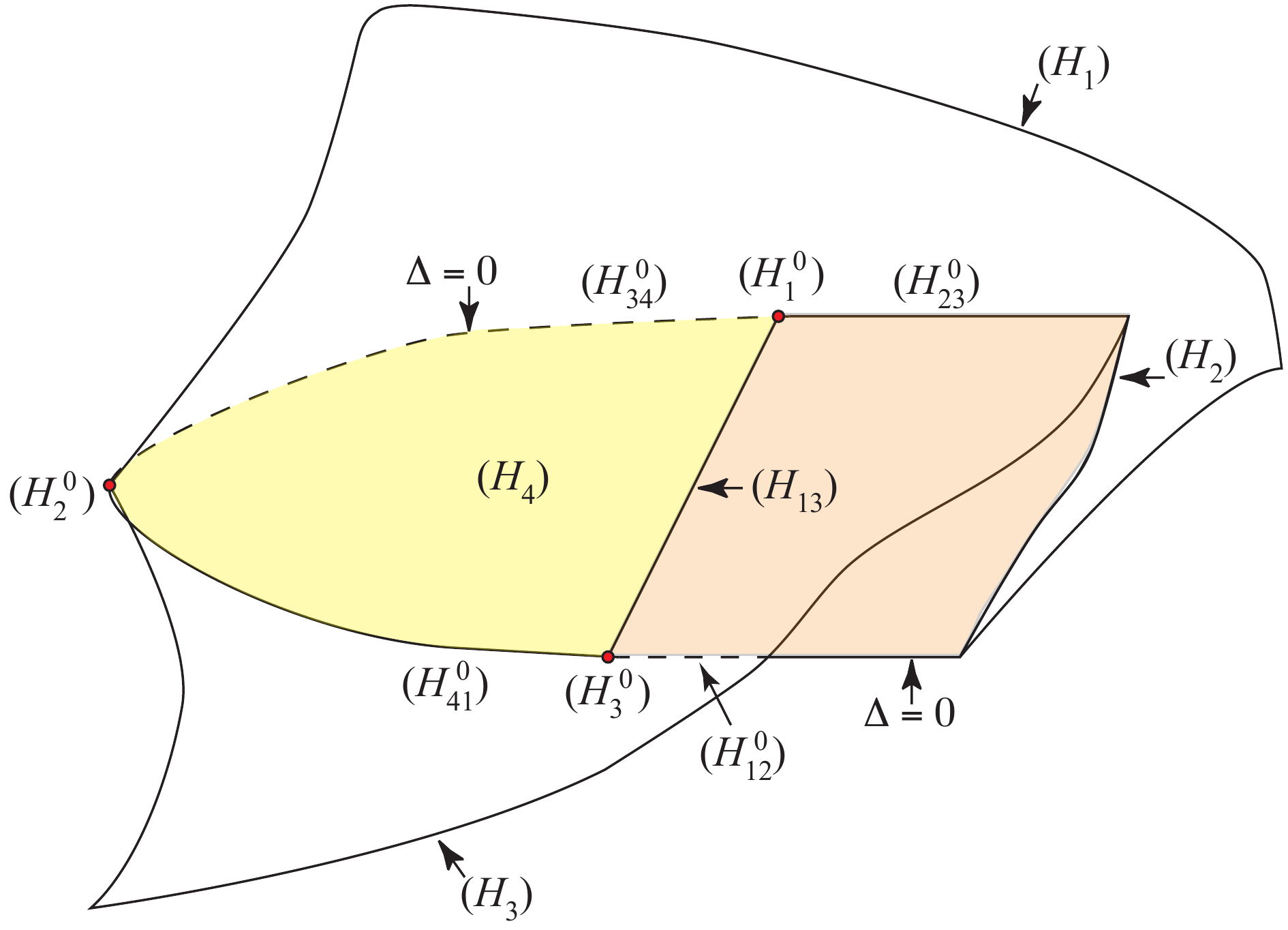} \caption{Gluing together two bifurcation diagrams for codimension $3$ points. The two surfaces $(H_1)$ and $(H_3)$ intersect along $(H_{13})$.}
  \label{fig:double}\end{center}
\end{figure}
This figure is obtained by linking together two copies of Figure~\ref{cod3} along $(H_{13})$. Now, we have \lq\lq free\rq\rq \ surfaces $(H_1)$ and $(H_3)$. Their missing boundary can only be along $(H_{24})$. In this figure we have that $(H_{34}^0)$ and $(H_{41}^0)$ have to merge at $(H_4^0)$, while $(H_{23}^0)$ and $(H_{12}^0)$ have to merge at $(H_2^0)$. 
The \lq\lq polyhedron\rq\rq\ of Figure~\ref{fig:pentagons} has eight vertices, ten edges and four faces. Four of the vertices are attached to three edges and four of them are artificial ones, since attached to only two edges. So the gluing can also be represented as that of an ordinary tetrahedron, with four of its edges cut in two parts as in Figure~\ref{fig:tetrahedron}. \end{proof}

\section{Description of the homoclinic bifurcations in the framework of Douady and Sentenac}\label{sec:DS}

\subsection{The construction of Douady and Sentenac} 
Let us consider a generic vector field for $\eps\in W_1$ or $\eps\in W_2$. The idea of Douady and Sentenac is that, when the time is extended to the complex domain, then all points of the phase space, except the singularities,  belong to  the \lq\lq complex trajectory\rq\rq\ of a single point. Hence, it is natural to reparameterize the phase space by the time $t$, where 
$$t(z)= \int_\infty^z \frac{d\zeta}{p_\eps(\zeta)},$$
but we have to pay attention since the  function $t(z)$ is multivalued. The phase plane minus the separatrices is the union of  two open simply connected regions (See Figure~\ref{fig:tau}). Each of these regions is the image of a strip in the  $t$-coordinate. The boundary of these strips are simply the inverse images of the separatrices (see Figure~\ref{fig:two_strips}), and the singular points have disappeared at infinity. Note that when a singular point is attached to a unique separatrix, then this separatrix is covered twice when sending the strip to the phase plane. The inverse image of $\infty $ is given by $4$ points, one on each boundary line of a strip (a pole can be reached in finite time).
\begin{figure}[h!]
\begin{center}
\subfigure[The two sectors]
{\includegraphics[width= 4cm]{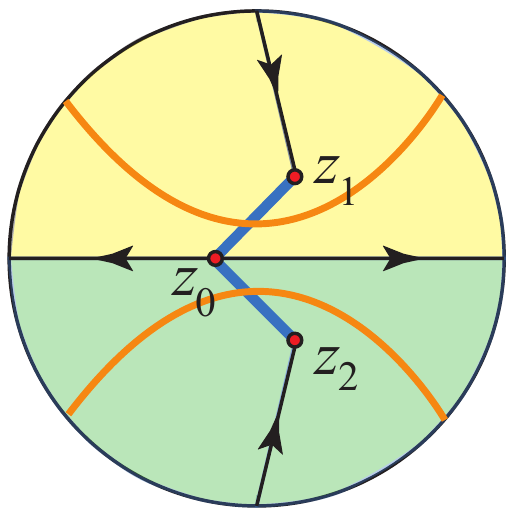}}\qquad
\subfigure[The corresponding strips in $t$-space]
{\includegraphics[width=7.5cm]{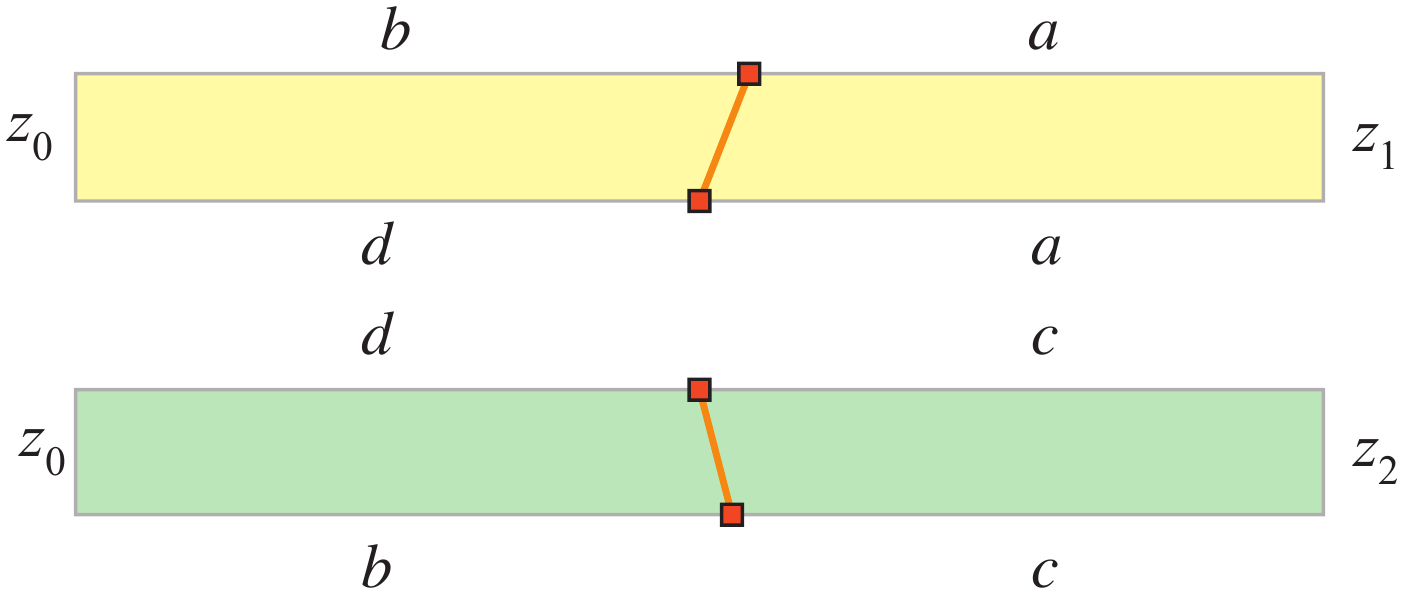}}
\caption{The two (infinite) strips for (a) (corresponding to (V) in Figure~\ref{passage_homocline}). The square dots represents the point at infinity and the strips need to be glued by pairing the identical letters to get a Riemann surface.} \label{fig:two_strips}\end{center} \end{figure}
The widths of the strips (between two images of $\infty$) are precisely given by the \lq\lq complex times\rq\rq\ to go from $\infty$ to $\infty$ without cutting the separatrices (see Figures~\ref{fig:tau}, \ref{fig:two_strips} and description in Section~\ref{sect:open_sets}). (Note that the residue theorem ensures that the time from $\infty$ to $\infty$ along a loop depends only on the homotopy class of the loop in $\CP^1\setminus\{z_1,z_2,z_3\}$.) 

The construction of Douady and Sentenac is then the following: they glue the strips according to the letters appearing in Figure~\ref{fig:two_strips}. This provides a Riemann surface which is exactly $\CP^1$ minus three points. If we put on the strips the constant vector field $\dot t =1$, then this endows $\CP^1$ with a cubic vector field with singular points at the three holes and a pole at infinity. This  vector field is unique up to affine changes of coordinates on the sphere, and hence has a unique presentation as $\dot z = z^3+\eps_1z+\eps_0$, up to $(z,\eps_1,\eps_0)\mapsto(-z,-\eps_1,\eps_0)$.   This is the spirit of the proof of Theorem~\ref{thm:DS}.

\subsection{Extending the construction of Douady and Sentenac to a bifurcation of homoclinic connection} 
Douady and Sentenac only described the generic case. We now complete their construction for the case of homoclinic connections, and explain the passage from $W_1$ to $W_2$ through a homoclinic connection. Figure~\ref{passage_homocline} represents schematically the four possible homoclinic bifurcations. (Note that the four homoclinic bifurcations do not occur for the same values of the parameters, so the singular points might have changed position at the time of the bifurcation: the figures should then be interpreted topologically.)
\begin{figure}[h!]\begin{center}
\includegraphics[width=8.4cm]{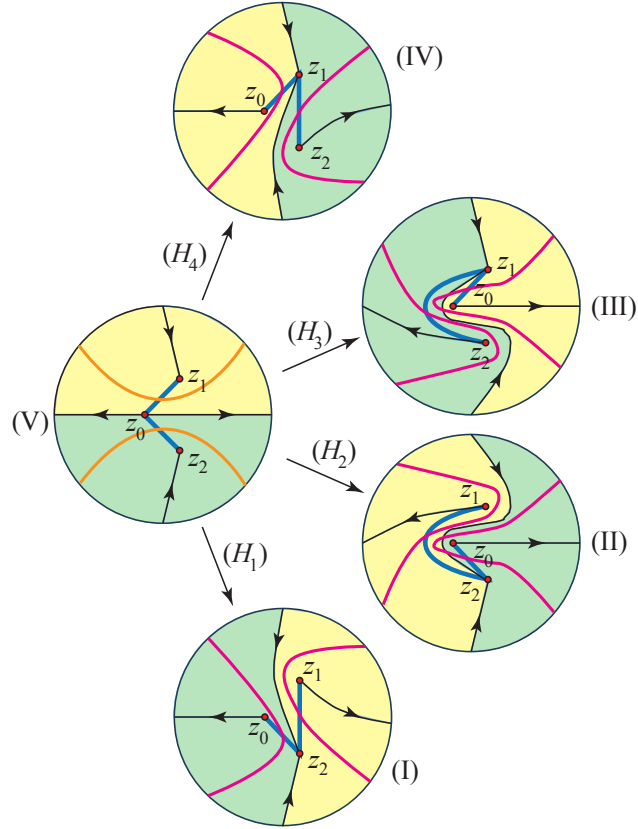}
\caption{In each case  we highlight the images of the two strips  inside a disk containing the singular points. The paths for the computations of the $\tau_{i,j}$ (see Notation~\ref{def:tau_ij}) have been drawn. They link quadrants 1 and 2 on one side, and 3 and 4 on the other side in (V), while they link quadrants 2 and 3, and 1 and 4 in (I)-(IV). The changes from $(\tau_{1,2}, \tau_{3,4})$ in (V) to $(\tau_{1,4}, \tau_{2,3})$ in (I)-(IV) are summarized in Table~\ref{Tab_change_tau}.} \label{passage_homocline} \end{center} \end{figure}

Let us concentrate on the bifurcation $(H_1)$: it can be approached from $W_2$ (corresponding to (V))  or from $W_1$ (corresponding to (I)):  Figures~\ref{fig: homoclinic_strip}  and \ref{fig: homoclinic_strip2} represent, for the approach on each $W_j$, the two regions in phase space limited by the separatrices,  and their representation by strips in $t$-space.
\begin{figure}[h!]
\begin{center}
\subfigure[The two regions for $(H_1)$]
{\includegraphics[width=4cm]{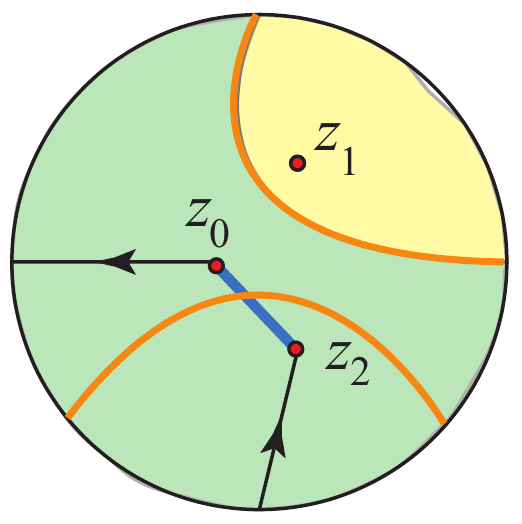}}\qquad
\subfigure[Point of view of $W_2$, i.e. (V)]
{\includegraphics[width=6cm]{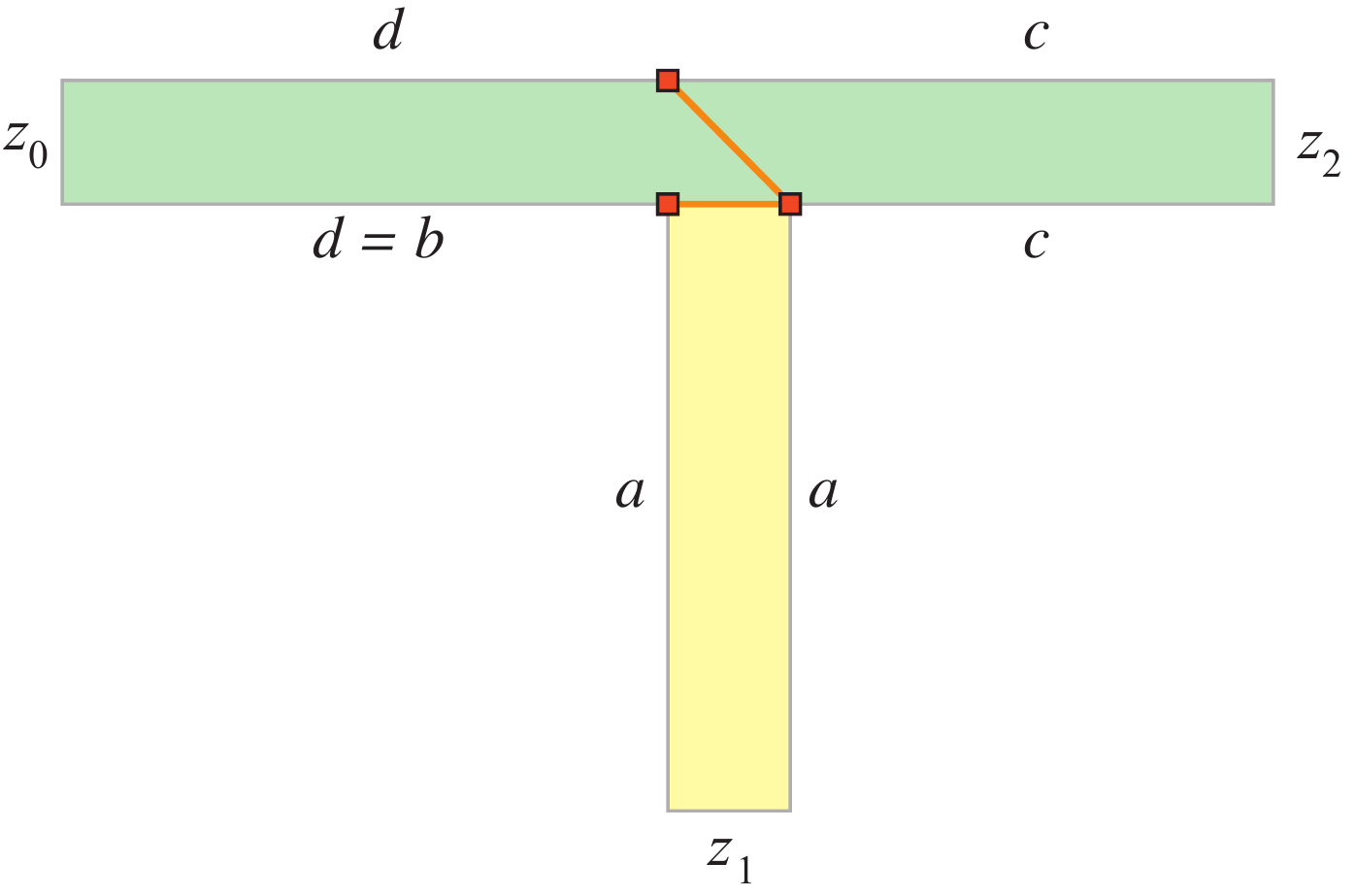}}
\caption{The homoclinic loop $(H_1)$ in (a), and the two strips representations seen from the point of view of $W_2$. } \label{fig: homoclinic_strip}\end{center} \end{figure}

 \begin{figure}[h!]
\begin{center}
\subfigure[The two regions for $(H_1)$]
{\includegraphics[width=4cm]{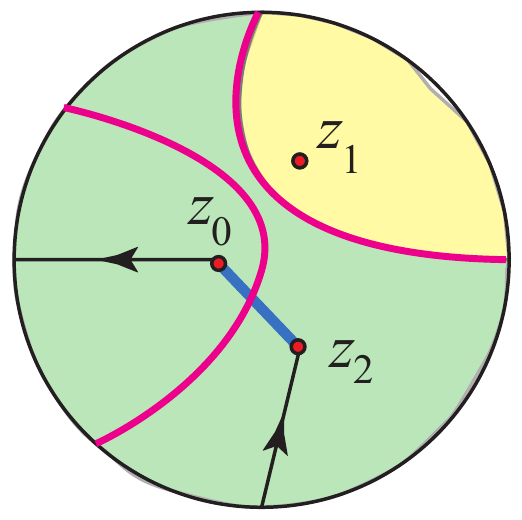}}\qquad
\subfigure[Point of view of $W_1$, i.e. (I)]
{\includegraphics[width=6cm]{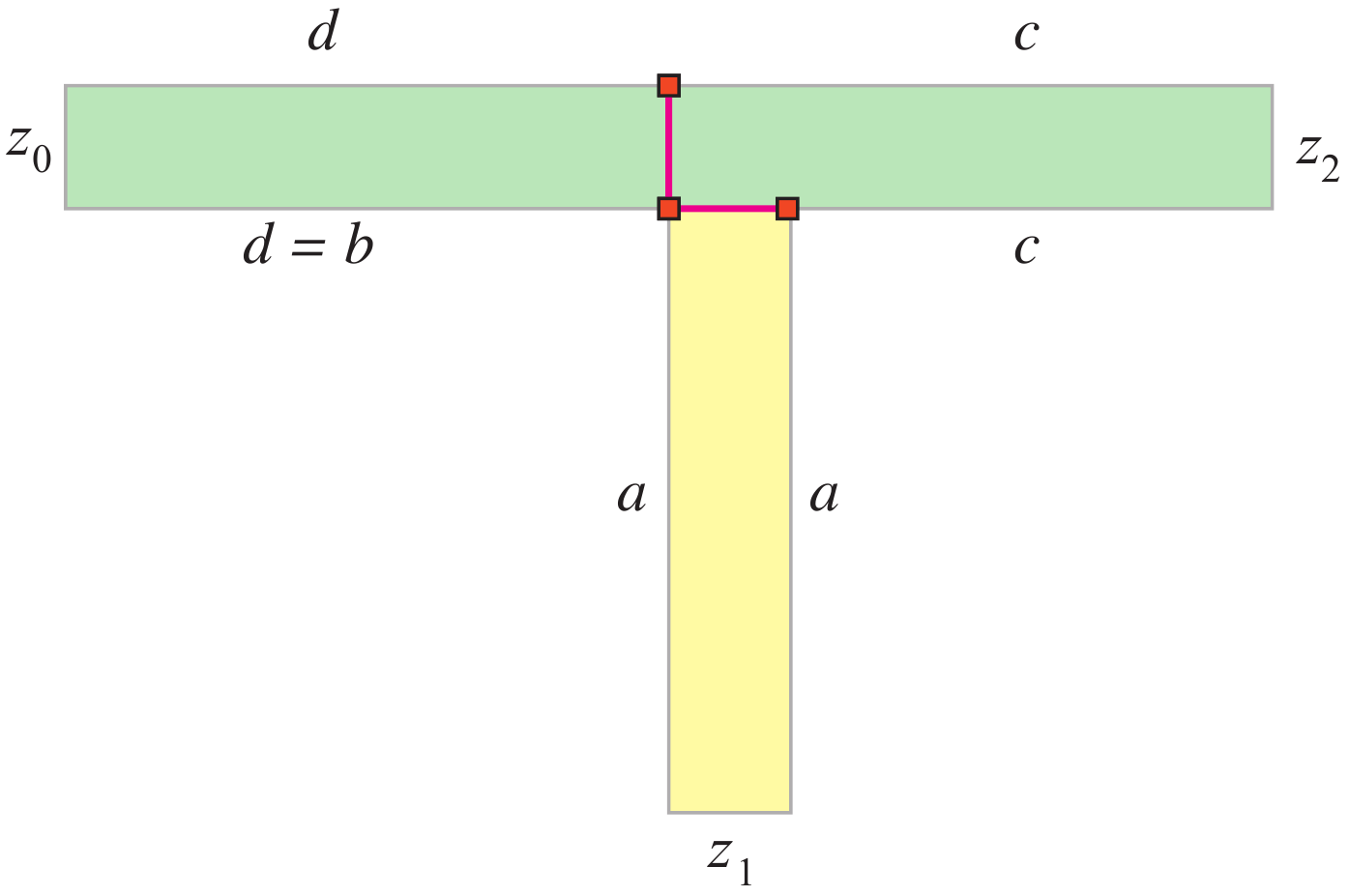}}\caption{The homoclinic loop $(H_1)$ in (a), and the two strips representations seen from $W_1$ in parameter space.} \label{fig: homoclinic_strip2}\end{center} \end{figure}

We see that we do not use the same $\tau$'s for both points of view. Hence, we introduce the following notation.

\begin{notation}\label{def:tau_ij}
Let $\tau_{i,j}$  be the travel time from quadrant $i$ to quadrant $j$ along curves from $\infty$ to $\infty$ as drawn in Figure~\ref{passage_homocline} and corresponding to a the segment transversal to a strip in Figure~\ref{fig:two_strips}. Hence, we have $\tau_{1,2}$ and $\tau_{3,4}$ for $\eps\in W_2$, and $\tau_{2,3} $ and $\tau_{1,4}$ for $\eps\in W_1$.\end{notation}

 A strip presentation of the passage through the homoclinic bifurcation $(H_1)$ appears in Figure~\ref{fig: homoclinic_strip}. The homoclinic loop surrounds a center. Suppose that the homoclinic bifurcation $(H_1)$ occurs for $\eps=\tilde{\eps}\in \ov{W_1}\cap\ov{W_2}$. All orbits inside the homoclinic loop $(H_1)$ have the same period, which we call $\tau_1$ and which is the limit of $\tau_{1,2}$ and $\tau_{1,4}$, when $\eps\to \tilde{\eps}$:
\begin{equation} \tau_1= \lim_{\substack{\eps\to \tilde{\eps}\\ \eps\in W_2}} \tau_{1,2} =- \lim_{\substack{\eps\to \tilde{\eps}\\ \eps\in W_1}} \tau_{1,4}.\label{eq_limite1}\end{equation} Hence, the set of closed orbits is represented by a vertical strip as in Figure~\ref{fig: homoclinic_strip}(b) or (c). Using that the $\tau$'s can be obtained from the residue theorem, it is clear that  
\begin{equation} \tau_{2,3}= \tau_{3,4}+\tau_1.\label{eq_limite2}\end{equation}

Figure~\ref{steps_passage} now presents the different steps of the passage from (V) (in $W_2$) to (I) (in $W_1$) through $(H_1)$. The imaginary part of $\tau_{1,2}$ decreases in (b) until it vanishes. So the left (horizontal) part of the yellow strip becomes thinner and thinner with a zero width at the limit in (c). At the same time the slope of the right part increases until it becomes vertical at $(H_1)$. When passing from (c) to (d), we only change from $(\tau_{1,2}, \tau_{3,4})$ to $(\tau_{1,4}, \tau_{2,3})$ through \eqref{eq_limite1} and \eqref{eq_limite2}. On the other side of $(H_1)$, in (e), the slope of the vertical part decreases (in absolute value) and bends to the left, while a thin half strip appears on the right. 

 \begin{figure}[H]
\begin{center}
\subfigure[Starting from (V) in $W_2$]
{\includegraphics[width=5.5cm]{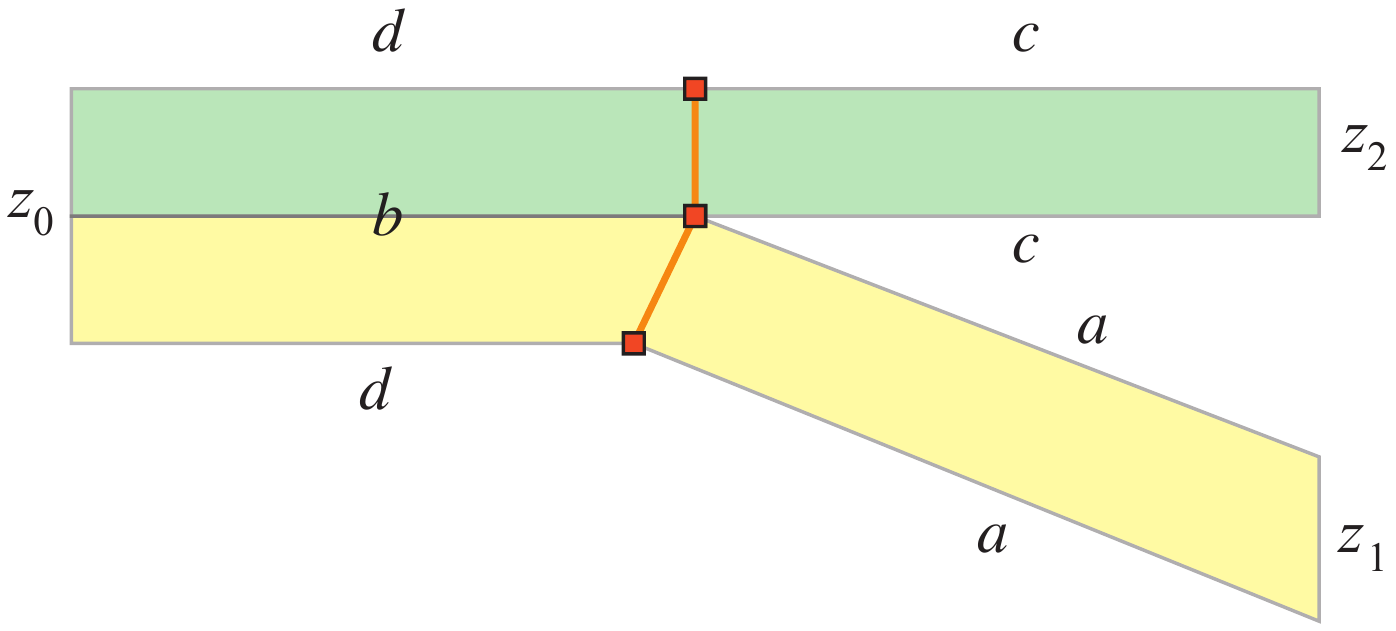}}\qquad
\subfigure[Approaching $(H_1)$ in $W_2$]
{\includegraphics[width=5.5cm]{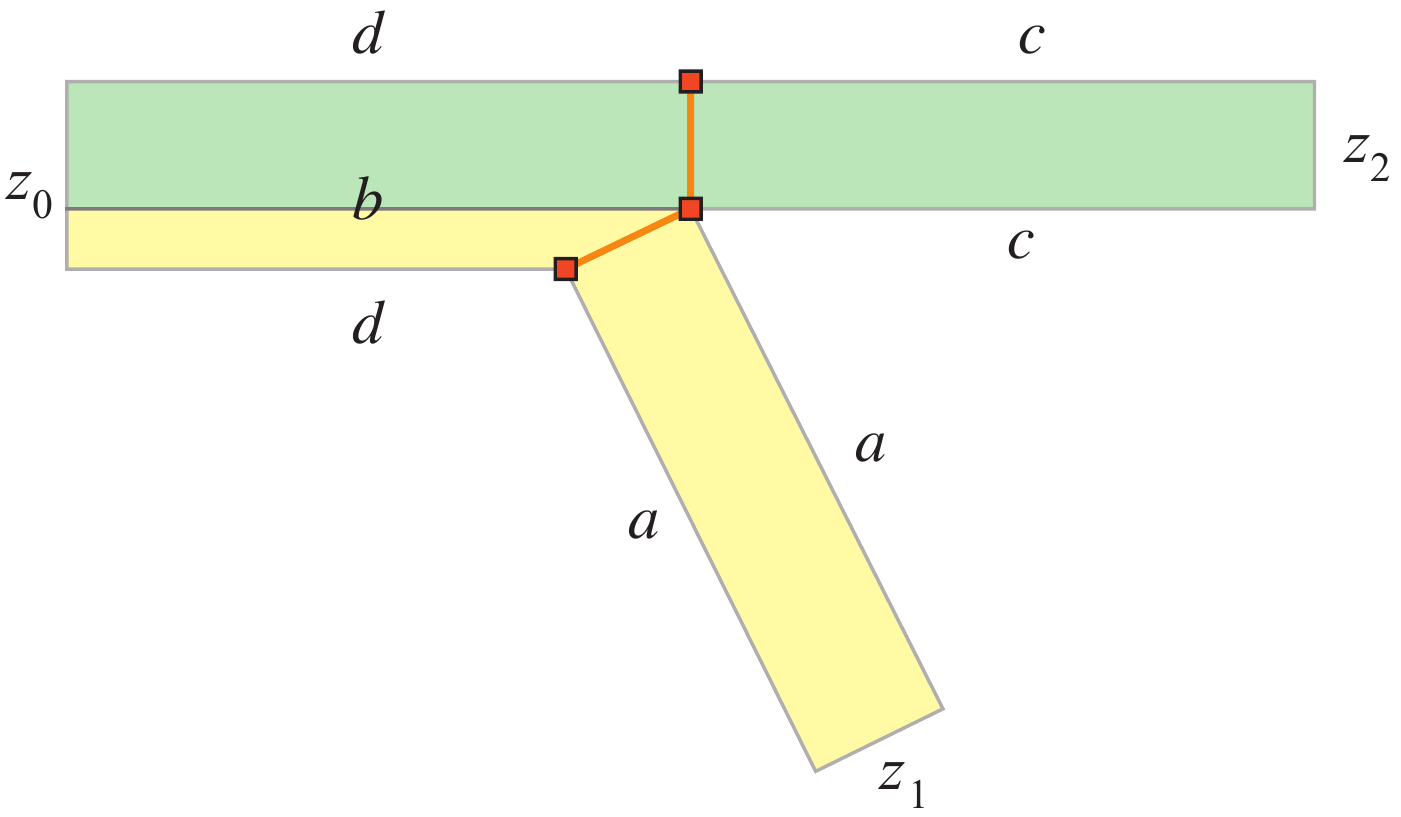}}\qquad
\subfigure[Homoclinic loop $(H_1)$ viewed from $W_2$] 
{\includegraphics[width=5.5cm]{HV_5_bis.pdf}}\qquad
\subfigure[Homoclinic loop $(H_1)$ viewed from $W_1$]
{\includegraphics[width=5.5cm]{HV_5_ter.pdf}}\qquad
\subfigure[Leaving the homoclinic loop in $W_1$]
{\includegraphics[width=5.5cm]{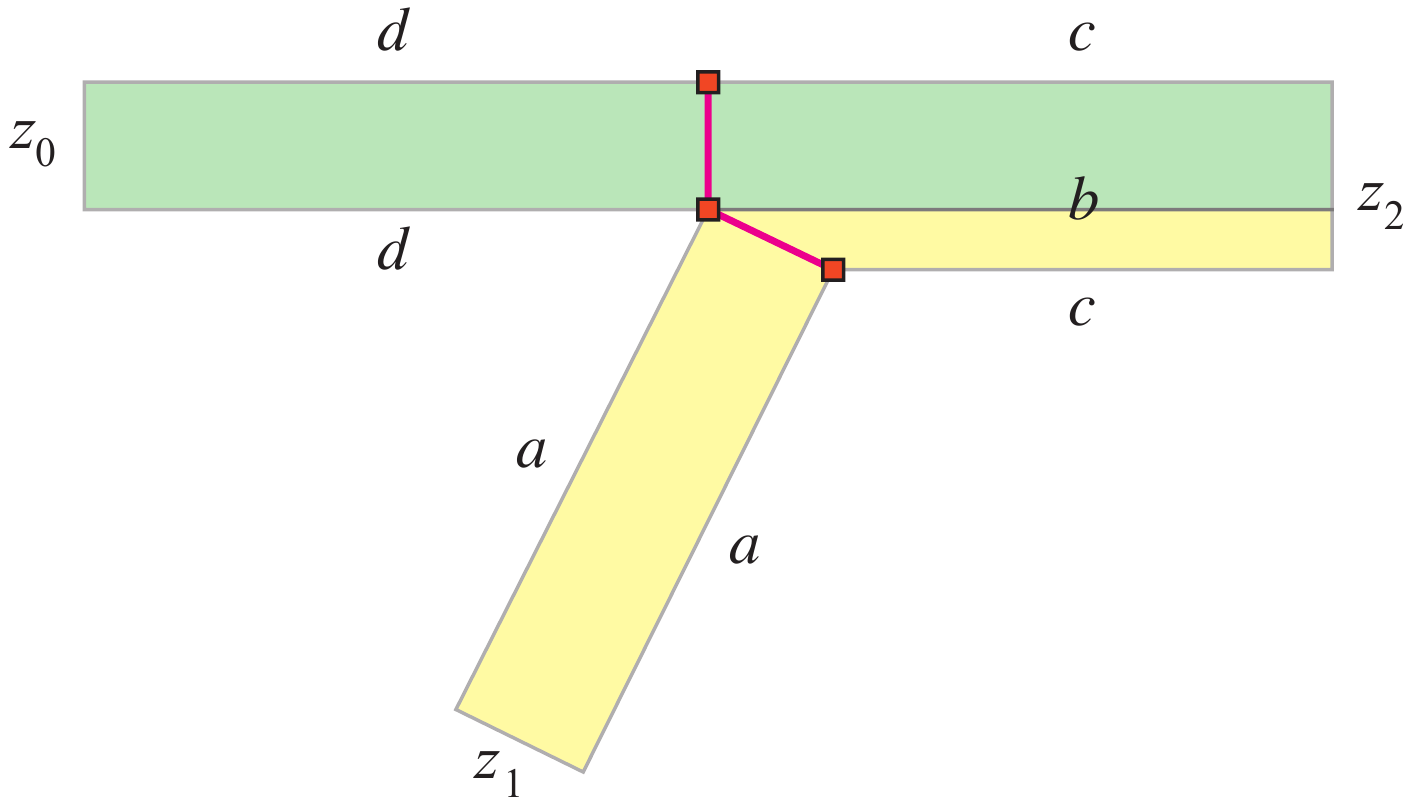}}\qquad
\subfigure[Ending in (I) in $W_1$]
{\includegraphics[width=5.5cm]{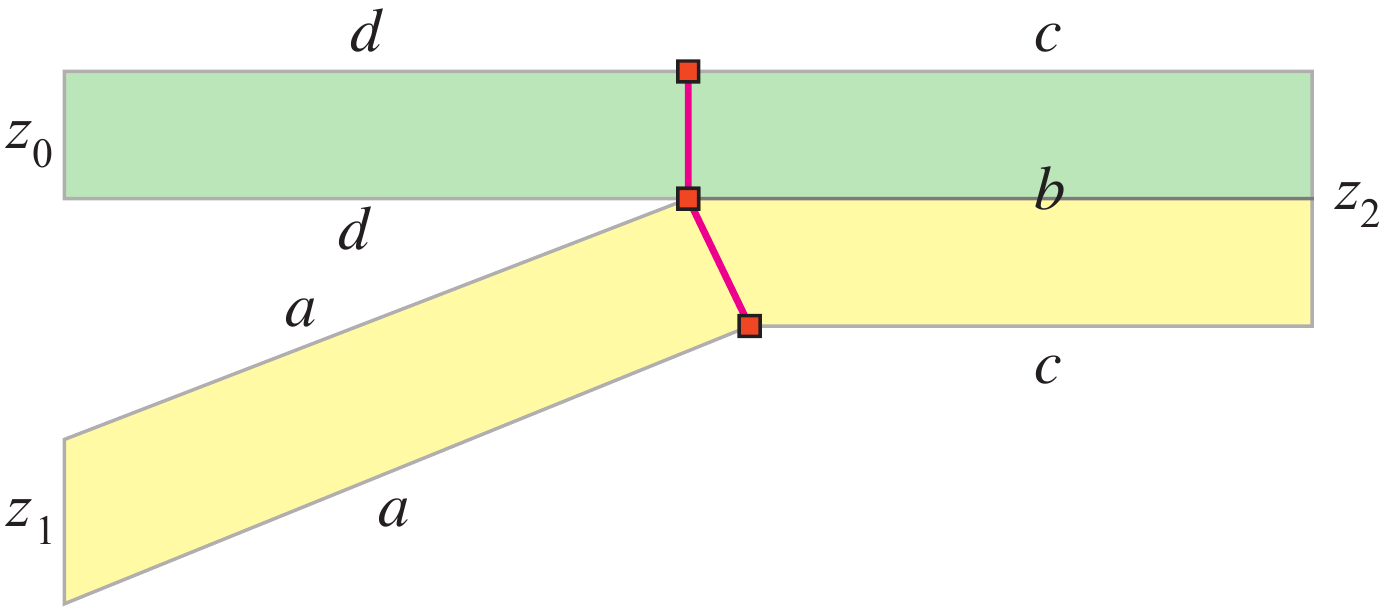}}
\caption{The different steps for passing from (V) to (I) through $(H_1)$.} \label{steps_passage}\end{center} \end{figure}

\begin{table}\label{Tab_change_tau}\begin{tabular}{|c||l|l|l|l|}
\hline
(V)&(I)&(II)&(III)&(IV)\\
\hline 
$\tau_{1,2}$&$\tau_{1,4}'= -\tau_{1,2}$&$\tau_{1,4}'=\tau_{1,2} +\tau_{3,4}$&$\tau_{1,4}'=\tau_{1,2} +\tau_{3,4}$&$\tau_{1,4}'=-\tau_{3,4}$\\
$\tau_{3,4}$&$\tau_{2,3}'=\tau_{1,2}+\tau_{3,4}$&$\tau_{2,3}'=-\tau_{1,2}$&$\tau_{2,3}'=-\tau_{3,4}$&$\tau_{2,3}'=\tau_{1,2}+\tau_{3,4}$\\
\hline
\end{tabular} \caption{Passing from $(\tau_{1,2}, \tau_{3,4})$ in (V) to $(\tau_{1,4}', \tau_{2,3}')$ in (I)-(IV).}\end{table}

 There are four possible homoclinic bifurcations $(H_i)$ occurring in the $i$-th quadrant to pass from $(V)$ to the corresponding case (I)-(IV). 
When passing from (V) to (I) or (II), then $\tau_{1,2}$ becomes real positive (resp. negative) while, when passing from (V) to (III) or (IV)), then $\tau_{3,4}$ becomes real negative (resp. positive). Note that the strips do not need to be horizontal in Figure~\ref{fig:two_strips}. They could be partly slanted as long as the width is given by the $\tau_{i,j}$ (see Figure~\ref{steps_passage}).

 \begin{figure}[H]
\begin{center}
\subfigure[Approaching $(H_1)$]
{\includegraphics[width=5.5cm]{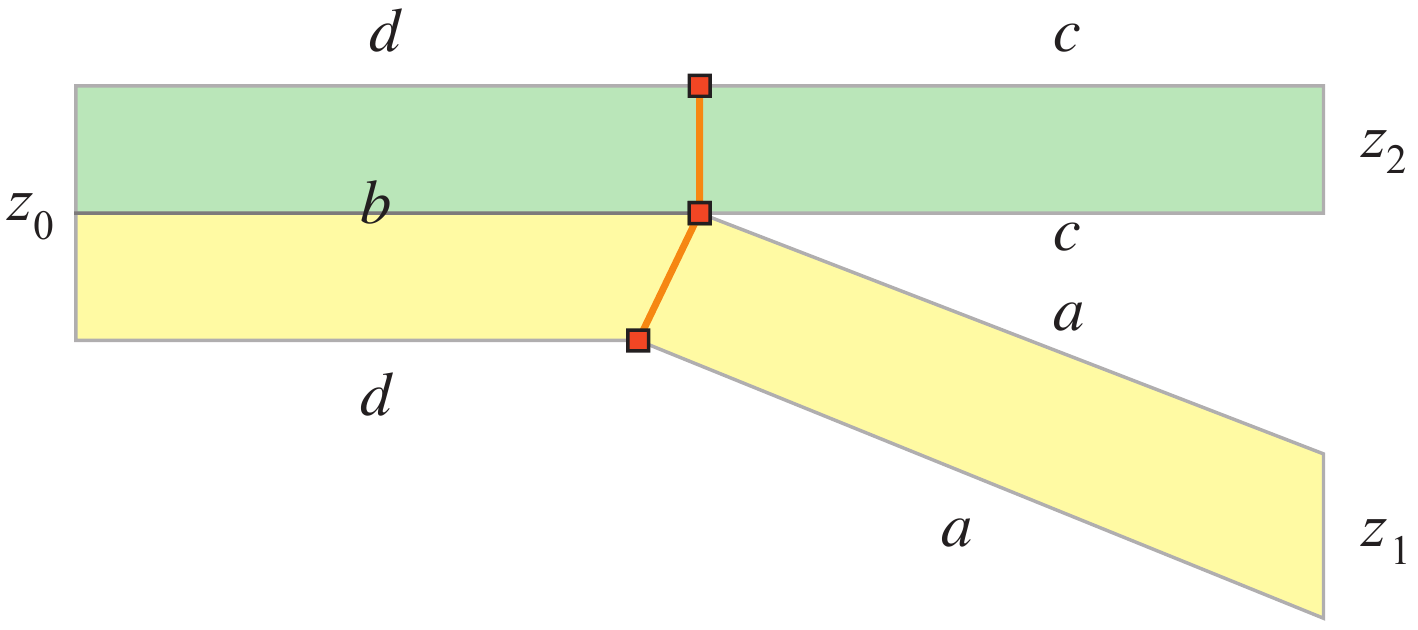}}\qquad
\subfigure[Approaching $(H_2)$]
{\includegraphics[width=5.5cm]{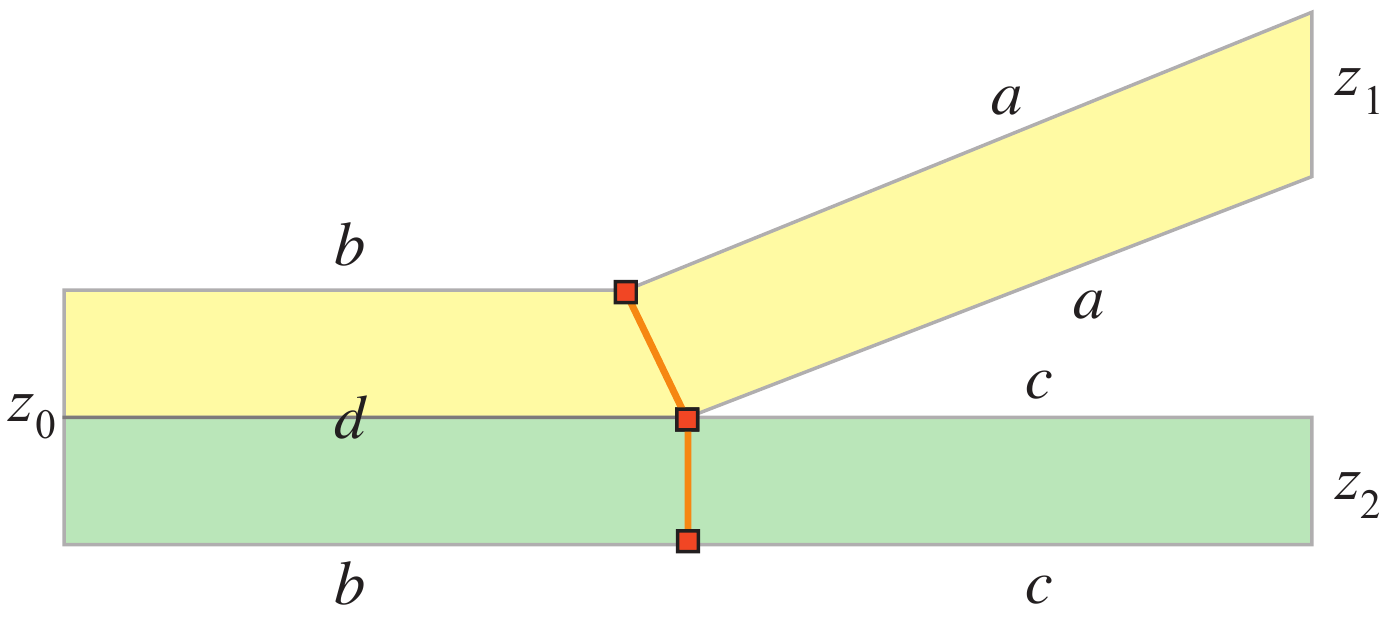}}\qquad
\subfigure[Approaching $(H_3)$]
{\includegraphics[width=5.5cm]{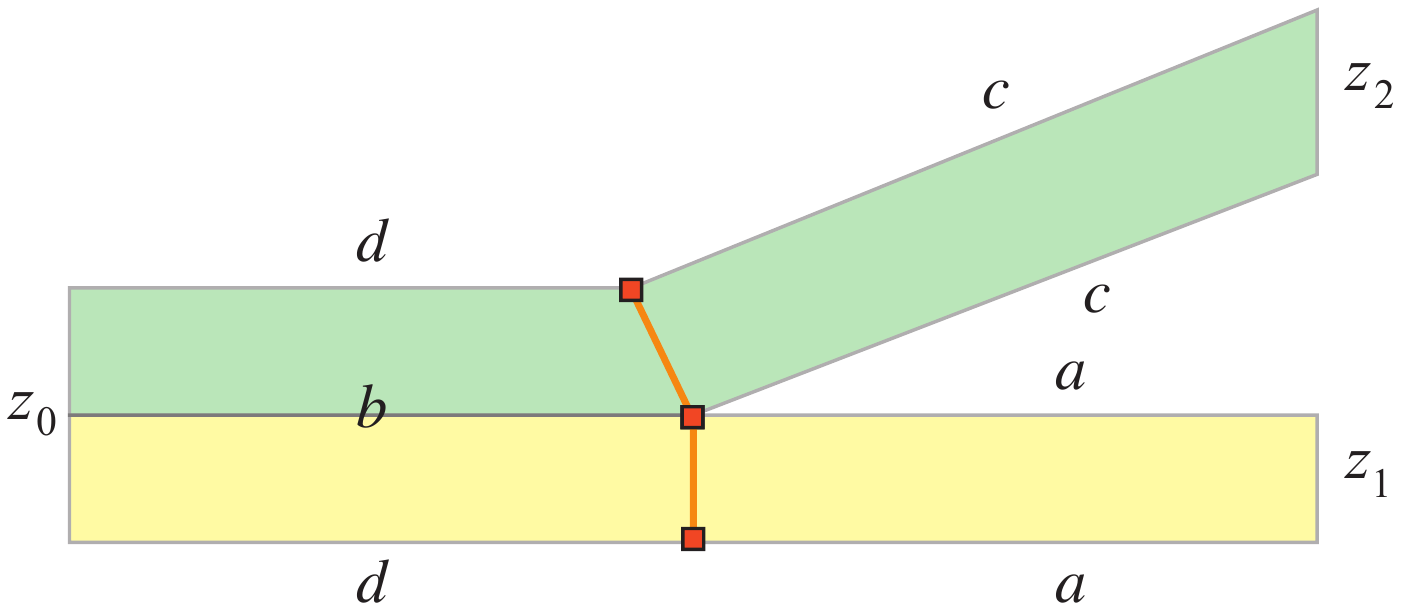}}\qquad
\subfigure[Approaching $(H_4)$]
{\includegraphics[width=5.5cm]{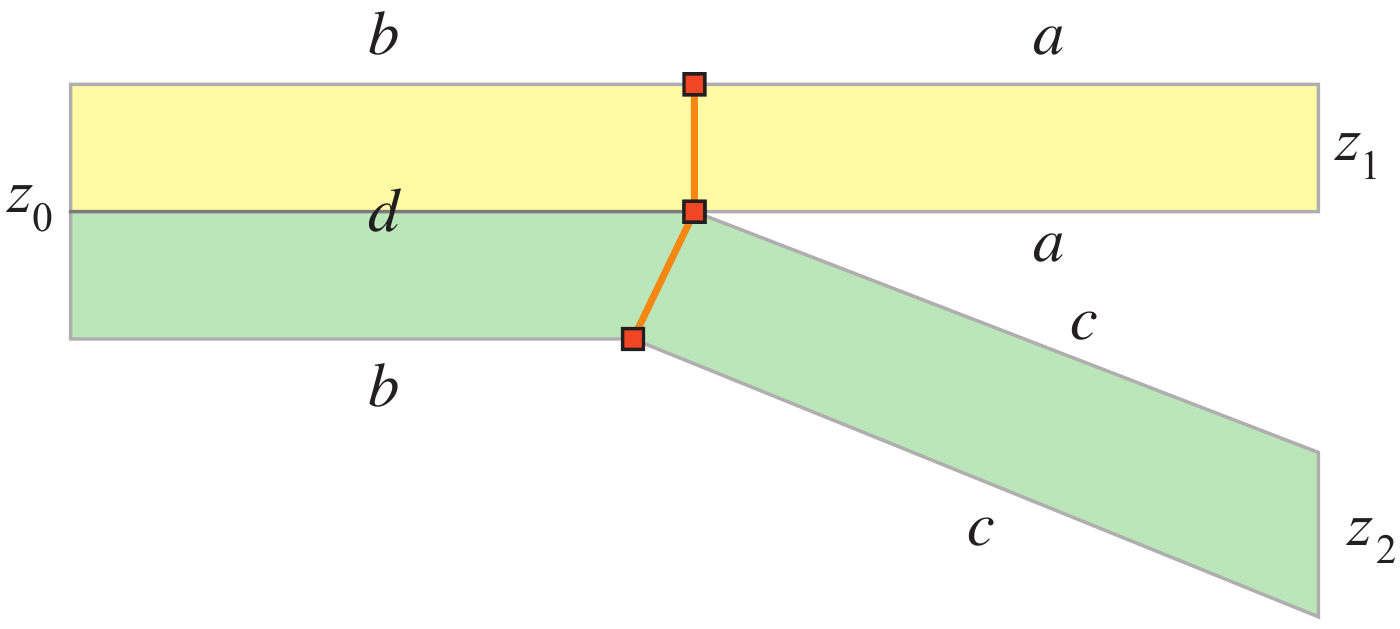}}\qquad
\caption{The four ways of seeing the gluings of the two strips for (V) to approach the four bifurcations of Figure~\ref{passage_homocline}.} \label{four_gluing}\end{center} \end{figure}

\section*{Acknowledgements}
The author is grateful to Colin Christopher and Jacques Hurtubise for helpful discussions.

\end{document}